\documentclass{amsart}
\usepackage{amssymb}
\usepackage{amsmath}
\usepackage{amsfonts}

\newtheorem{theorem}{Theorem}
\theoremstyle{plain}

\newtheorem{claim}{Claim}

\newtheorem{corollary}{Corollary}

\newtheorem{definition}{Definition}
\newtheorem{example}{Example}

\newtheorem{lemma}{Lemma}

\newtheorem{proposition}{Proposition}
\newtheorem{remark}{Remark}

\numberwithin{equation}{section}
\input{tcilatex}

\begin{document}
\title{Hyperconvergence in topological dynamics}
\author{Josiney A. Souza}
\author{Richard W. M. Alves}
\address{Universidade Estadual de Maring\'{a}, Brazil. \\
Email: joasouza3@uem.br}

\begin{abstract}
In this manuscript the concept of hyperspace is revisited. The main purpose
is to study hyperconvergence and continuity of orbital and limit set
functions for semigroup action on completely regular space. Some general
facts on Hausdorff and Kuratowski hyperconvergence are presented.
\end{abstract}

\maketitle

\section{Introduction}

The present paper contributes to topological dynamics by studying
hyperconvergence and continuity of orbital and limit set functions for
semigroup actions on topological spaces. The phase space is required to be
completely regular for the purpose of selecting an admissible family of open
coverings. The admissible structure is employed to reproduce a uniformity
for hyperspace and define hyperconvergence and semicontinuity of set-valued
functions.

One of the main theorems in topology establishes the Hausdorff metric on the
set of all compact subsets of a metric space. E. Michael \cite{Michael}
reproduced the Hausdorff topology in the setting of compact spaces under
absence of metrization. The Hausdorff topology is currently a special
concept of the hyperspace theory on uniformizable spaces (e.g. \cite%
{Nadler,Will}). Inspired in the original metric space methods, we reproduce
a uniform structure on hyperspace by means of a binary function that
generalizes the Hausdorff metric. The strategy is to consider the
description of uniformizable spaces as admissible spaces, which are
topological spaces admitting admissible family of open coverings (\cite%
{Richard}). If $X$ is an admissible space endowed with an admissible family
of open coverings $\mathcal{O}$, it is possible to define a binary function $%
\rho _{H}$ on the hyperspace $\mathcal{H}\left( X\right) $ with values in
the powerset $\mathcal{P}\left( \mathcal{O}\right) $. This function is
employed to construct the surroundings for a base of diagonal uniformity on $%
\mathcal{H}\left( X\right) $ (Theorem \ref{T4}). The \textquotedblleft
balls\textquotedblright\ of $\rho _{H}$ form a base for a uniform topology
on $\mathcal{H}\left( X\right) $ (Theorem \ref{T3}). This uniform structure
yields the extension of classical results on hyperconvergence. In special,
we relate the Kuratowski hyperconvergence and the Hausdorff hyperconvergence
(Theorems \ref{PK} and \ref{KP}).

The main purpose of the paper is to study hyperconvergence in topological
dynamics. We follow the line of investigation of the papers \cite%
{BragaSouza,BBRS,BBRSCan} inspired by S. Saperstone and M. Nishihama \cite%
{SA} from studies of stability and continuity of orbital and limit set maps
of semiflows on metric spaces. For introducing the ideas of the paper, let $%
S $ be a semigroup acting on the admissible space $X$ and let $\mathcal{F}$
be a filter basis on the subsets of $S$. For a given $x\in X$, the $\omega $%
-limit set and the prolongational limit set of $x$ on the direction of $%
\mathcal{F}$ are respectively defined by
\begin{equation*}
L_{\mathcal{F}}\left( x\right) =\bigcap_{A\in \mathcal{F}}K_{A}\left(
x\right) \qquad \qquad J_{\mathcal{F}}\left( x\right) =\bigcap_{A\in
\mathcal{F}}D_{A}\left( x\right)
\end{equation*}%
where $K_{A}\left( x\right) =\mathrm{cls}\left( Ax\right) $ and $D_{A}\left(
x\right) =\bigcap_{\mathcal{U}\in \mathcal{O}}\mathrm{cls}\left( A\mathrm{St}%
\left[ x,\mathcal{U}\right] \right) $ (\cite{BragaSouza,BBRS,SouzaTozatti}).
It is intuitive that the nets $\left( K_{A}\left( x\right) \right) _{A\in
\mathcal{F}}$ and $\left( D_{A}\left( x\right) \right) _{A\in \mathcal{F}}$
in $\mathcal{H}\left( X\right) $ converge respectively to the limit set $L_{%
\mathcal{F}}\left( x\right) $ and to the prolongational limit set $J_{%
\mathcal{F}}\left( x\right) $. We prove these facts under certain conditions
(Propositions \ref{P4} and \ref{P5}). In other words, the nets of set-valued
functions $\left( K_{A}\right) _{A\in \mathcal{F}}$ and $\left( D_{A}\left(
x\right) \right) _{A\in \mathcal{F}}$ pointwise converge to the functions $%
L_{\mathcal{F}}$ and $J_{\mathcal{F}}$, respectively.

In line of these statements, we investigate the Hausdorff continuity of the
set-valued functions $K_{A}$, $D_{A}$, $L_{\mathcal{F}}$, and $J_{\mathcal{F}%
}$. The Hausdorff continuity of $K_{A}$ means that the $A$-orbits have
trivial prolongations, that is, $K_{A}=D_{A}$ (Proposition \ref{K}). If $%
K_{A}$ is Hausdorff continuous for all $A\in \mathcal{F}$, it follows that $%
L_{\mathcal{F}}=J_{\mathcal{F}}$. In special case, the Hausdorff continuity
of $K_{S}$ means the stability of the orbit closure $\mathrm{cls}\left(
Sx\right) $ for every $x\in X$ (Theorem \ref{T1}). In general, under certain
conditions, the Hausdorff continuity of $L_{\mathcal{F}}$ means that every
limit set $L_{\mathcal{F}}\left( x\right) $ is eventually stable (Theorem %
\ref{T10}).

The results of the paper extend to general topological spaces under absence
of metrization. For instance, if $G$ is a topological group and $S\subset G$
is a closed subgroup of $G$ then the function $K_{S}:G\rightarrow \mathcal{H}%
\left( G\right) $ assumes values in $\left. G\right/ S$ and then it
corresponds to the standard projection $\pi :G\rightarrow \left. G\right/ S$%
. A natural question is the following: by inducing a uniformity on $\left.
G\right/ S$ from the uniformity of $\mathcal{H}\left( G\right) $, does the
corresponding uniform topology coincide with the quotient topology? We have
a positive answer for this question (Theorem \ref{T12}). This means that $%
K_{S}:G\rightarrow \mathcal{H}\left( G\right) $ is Hausdorff continuous and,
in the case $S$ compact, every left coset $gS$ is stable.

For another nonmetrizable example, let $E^{E}$ be the function space of a
normed vector space $E$ endowed with the uniformity of pointwise
convergence. Let $F\subset E$ be a compact set and $X\subset E^{E}$ the
subspace of all contraction maps of $E$ with fixed point in $F$ and same
Lipschitz constant $L<1$. The subset $i\left( F\right) \subset X$ is the
global asymptotically stable set for the action of the multiplicative
positive integers on $X$ given by $nf=f^{n}$. The limit set function $L_{%
\mathcal{F}}$ is Hausdorff continuous on $X$, essentially because $L_{%
\mathcal{F}}\left( f\right) =i\left( \mathrm{fix}f\right) $ (Example \ref%
{Ex4}).

The paper is organized as follows. In Section \ref{Section1} we recall some
definitions and fix notations of admissible structure on completely regular
space. In Section \ref{SectionHyperspace} we describe hyperspace by means of
a binary function on the admissible space. The notions of hyperconvergence
are defined in Section \ref{HK}, where we reproduce some classical theorems
involving Kuratowski hyperconvergence and Hausdorff hyperconvergence. The
main results of the paper are present in Section \ref{s4} where we apply the
results on hyperconvergence to study semicontinuity and Hausdorff continuity
of functions defined by orbit closure, limit set, prolongation, and
prolongational limit set. In the last section, we provide illustrating
examples for the theory presented in the paper.

\section{\label{Section1}Admissible structure}

This section contains the basic definitions and properties of admissible
spaces. We refer to \cite{Richard}, \cite{patrao2}, and \cite{So} for the
previous development of admissible spaces.

Let $X$ be a topological space and $\mathcal{U},\mathcal{V}$ coverings of $X$%
. We write $\mathcal{V}\leqslant \mathcal{U}$ if $\mathcal{V}$ is a
refinement of $\mathcal{U}$. One says $\mathcal{V}$ double-refines $\mathcal{%
U}$, or $\mathcal{V}$ is a double-refinement of $\mathcal{U}$, written $%
\mathcal{V}\leqslant \frac{1}{2}\mathcal{U}$ or $2\mathcal{V}\leqslant
\mathcal{U}$, if for every $V,V^{\prime }\in \mathcal{V}$, with $V\cap
V^{\prime }\neq \emptyset $, there is $U\in \mathcal{U}$ such that $V\cup
V^{\prime }\subset U$. We write $\mathcal{V}\leqslant \frac{1}{2^{2}}%
\mathcal{U}$ if there is a covering $\mathcal{W}$ of $X$ such that $\mathcal{%
V}\leqslant \frac{1}{2}\mathcal{W}$ and $\mathcal{W}\leqslant \frac{1}{2}%
\mathcal{U}$. Inductively, we write $\mathcal{V}\leqslant \frac{1}{2^{n}}%
\mathcal{U}$ if there is $\mathcal{W}$ with $\mathcal{V}\leqslant \frac{1}{2}%
\mathcal{W}$ and $\mathcal{W}\leqslant \frac{1}{2^{n-1}}\mathcal{U}$. In
certain sense, the notion of double-refinement in topological spaces
corresponds to the property of triangle inequality in metric spaces.

Now, for a covering $\mathcal{U}$ of $X$ and a subset $Y\subset X$, the
\emph{star} of $Y$ with respect to $\mathcal{U}$ is the set
\begin{equation*}
\mathrm{St}\left[ Y,\mathcal{U}\right] =\bigcup \left\{ U\in \mathcal{U}%
:Y\cap U\neq \emptyset \right\} \text{.}
\end{equation*}%
If $Y=\left\{ x\right\} $, we usually write $\mathrm{St}\left[ x,\mathcal{U}%
\right] $ rather than $\mathrm{St}\left[ \left\{ x\right\} ,\mathcal{U}%
\right] $. Then one has $\mathrm{St}\left[ Y,\mathcal{U}\right]
=\bigcup\limits_{x\in Y}\mathrm{St}\left[ x,\mathcal{U}\right] $ for every
subset $Y\subset X$.

\begin{definition}
\label{Admiss} A family $\mathcal{O}$ of open coverings of $X$ is said to be
\textbf{admissible} if it satisfies the following properties:

\begin{enumerate}
\item For any $\mathcal{U},\mathcal{V}\in \mathcal{O}$, there is $\mathcal{W}%
\in \mathcal{O}$ such that $\mathcal{W}\leqslant \frac{1}{2}\mathcal{U}$ and
$\mathcal{W}\leqslant \frac{1}{2}\mathcal{V}$.;

\item The stars $\mathrm{St}\left[ x,\mathcal{U}\right] $, for $x\in X$ and $%
\mathcal{U}\in \mathcal{O}$, form a basis for the topology of $X$.
\end{enumerate}

The space $X$ is called \textbf{admissible} if it admits an admissible
family of open coverings.
\end{definition}

\begin{example}
\label{Ex1}

\begin{enumerate}
\item If $X$ is a paracompact Hausdorff space, then the family $\mathcal{O}$
of all open coverings of $X$ is admissible.

\item If $X$ is a compact Hausdorff space, then the family $\mathcal{O}_{f}$
of all finite open coverings of $X$ is admissible.

\item If $\left( X,\mathrm{d}\right) $ is a pseudometric space, then the
family $\mathcal{O}_{\mathrm{d}}$ of the coverings $\mathcal{U}_{\varepsilon
}=\left\{ \mathrm{B}_{\mathrm{d}}\left( x,\varepsilon \right) :x\in
X\right\} $ by $\varepsilon $-balls, for $\varepsilon >0$, is admissible.
For every $\varepsilon >0$ and $Y\subset X$ we have $\mathcal{U}%
_{\varepsilon /2}\leqslant \frac{1}{2}\mathcal{U}_{\varepsilon }$ and
\begin{equation*}
\mathrm{St}\left[ Y,\mathcal{U}_{\varepsilon /2}\right] \subset \mathrm{B}_{%
\mathrm{d}}\left( Y,\varepsilon \right) \subset \mathrm{St}\left[ Y,\mathcal{%
U}_{\varepsilon }\right] \text{.}
\end{equation*}

\item If $X$ is a uniformizable space then any covering uniformity of $X$ is
an admissible family of open coverings of $X$.
\end{enumerate}
\end{example}

\begin{remark}
\label{R1} Since the collection $\left\{ \mathrm{St}\left[ x,\mathcal{U}%
\right] :\mathcal{U}\in \mathcal{O}\right\} $ is a neighborhood base at $%
x\in X$, one has $\bigcap\limits_{\mathcal{U}\in \mathcal{O}}\mathrm{St}%
\left[ Y,\mathcal{U}\right] =\mathrm{cls}\left( Y\right) $ for every subset $%
Y\subset X$. If $X$ is Hausdorff, it follows that $\bigcap\limits_{\mathcal{U%
}\in \mathcal{O}}\mathrm{St}\left[ x,\mathcal{U}\right] =\left\{ x\right\} $
for every $x\in X$.
\end{remark}

\begin{remark}
If $K\subset X$ is compact and $V\subset X$ is open with $K\subset V$ then
there is $\mathcal{U}\in \mathcal{O}$ such that $\mathrm{St}\left[ K,%
\mathcal{U}\right] \subset V$.
\end{remark}

A topological space $X$ is admissible if and only if it is uniformizable (%
\cite{Richard}). It is well-known that $X$ is uniformizable if and only if
it is completely regular.

Let $X$ be a fixed completely regular space endowed with an admissible
family of open coverings $\mathcal{O}$. Let $\mathcal{P}\left( \mathcal{O}%
\right) $ denote the power set of $\mathcal{O}$ and consider the partial
ordering relation on $\mathcal{P}\left( \mathcal{O}\right) $ given by
inverse inclusion: for $\mathcal{E}_{1},\mathcal{E}_{2}\in \mathcal{P}\left(
\mathcal{O}\right) $%
\begin{equation*}
\mathcal{E}_{1}\prec \mathcal{E}_{2}\text{ if and only if }\mathcal{E}%
_{1}\supset \mathcal{E}_{2}.
\end{equation*}%
Concerning this relation, $\mathcal{O}$ is the lower bound for $\mathcal{P}%
\left( \mathcal{O}\right) $ (the \textquotedblleft zero\textquotedblright )
and the empty set $\emptyset $ is the upper bound for $\mathcal{P}\left(
\mathcal{O}\right) $ (the \textquotedblleft infinity\textquotedblright ).

For each $\mathcal{E}\in \mathcal{P}\left( \mathcal{O}\right) $ and $n\in
\mathbb{N}^{\ast }$ we define the set $n\mathcal{E}$ in $\mathcal{P}\left(
\mathcal{O}\right) $ by
\begin{equation*}
n\mathcal{E}=\left\{ \mathcal{U}\in \mathcal{O}:\text{there is }\mathcal{V}%
\in \mathcal{E}\text{ such that }\mathcal{V}\leqslant \tfrac{1}{2^{n}}%
\mathcal{U}\right\} .
\end{equation*}

This operation is order-preserving, that is, if $\mathcal{E}\prec \mathcal{D}
$ then $n\mathcal{E}\prec n\mathcal{D}$. In fact, if $\mathcal{U}\in n%
\mathcal{D}$ then there is $\mathcal{V}\in \mathcal{D}$ such that $\mathcal{V%
}\leqslant \tfrac{1}{2^{n}}\mathcal{U}$. As $\mathcal{D}\subset \mathcal{E}$%
, it follows that $\mathcal{U}\in n\mathcal{E}$, and therefore $n\mathcal{E}%
\prec n\mathcal{D}$. Note also that $n\mathcal{O}=\mathcal{O}$, for every $%
n\in \mathbb{N}^{\ast }$, since for each $\mathcal{U}\in \mathcal{O}$ there
is $\mathcal{V}\in \mathcal{O}$ such that $\mathcal{V}\leqslant \tfrac{1}{%
2^{n}}\mathcal{U}$, that is, $\mathcal{U}\in n\mathcal{O}$.

We often consider the following notion of convergence in $\mathcal{P}\left(
\mathcal{O}\right) $.

\begin{definition}
\label{Convergence}We say that a net $\left( \mathcal{E}_{\lambda }\right) $
in $\mathcal{P}\left( \mathcal{O}\right) $ \textbf{converges} to $\mathcal{O}
$, written $\mathcal{E}_{\lambda }\rightarrow \mathcal{O}$,\ if for every $%
\mathcal{U}\in \mathcal{O}$ there is a $\lambda _{0}$ such that $\mathcal{U}%
\in \mathcal{E}_{\lambda }$ whenever $\lambda \geq \lambda _{0}$.
\end{definition}

It is easily seen that $\mathcal{D}_{\lambda }\prec \mathcal{E}_{\lambda }$
and $\mathcal{E}_{\lambda }\rightarrow \mathcal{O}$ implies $\mathcal{D}%
_{\lambda }\rightarrow \mathcal{O}$. Moreover, $\mathcal{E}_{\lambda
}\rightarrow \mathcal{O}$ implies $n\mathcal{E}_{\lambda }\rightarrow
\mathcal{O}$ for every $n\in \mathbb{N}^{\ast }$ (see \cite[Proposition 1]%
{RichardA}).

We also need the auxiliary function $\rho :X\times X\rightarrow \mathcal{P}%
\left( \mathcal{O}\right) $ given by
\begin{equation*}
\rho \left( x,y\right) =\left\{ \mathcal{U}\in \mathcal{O}:y\in \mathrm{St}%
\left[ x,\mathcal{U}\right] \right\} .
\end{equation*}%
Note that the value $\rho \left( x,y\right) $ is upwards hereditary, that
is, if $\mathcal{U}\leqslant \mathcal{V}$ with $\mathcal{U}\in \rho \left(
x,y\right) $ then $\mathcal{V}\in \rho \left( x,y\right) $. The following
properties of the function $\rho $ are proved in \cite[Propositions 2 and 3]%
{RichardA}.

\begin{proposition}
\label{P1}

\begin{enumerate}
\item $\rho \left( x,y\right) =\rho \left( y,x\right) $ for all $x,y\in X$.

\item $\mathcal{O}\prec \rho \left( x,y\right) $, for all $x,y\in X$, and $%
\mathcal{O}=\rho \left( x,x\right) $.

\item If $X$ is Hausdorff, $\mathcal{O}=\rho \left( x,y\right) $ if and only
if $x=y$.

\item $\rho \left( x,y\right) \prec n\left( \rho \left( x,x_{1}\right) \cap
\rho \left( x_{1},x_{2}\right) \cap \ldots \cap \rho \left( x_{n},y\right)
\right) $ for all $x,y,x_{1},...,x_{n}\in X$.

\item A net $\left( x_{\lambda }\right) $ in $X$ converges to $x$ if and
only if $\rho \left( x_{\lambda },x\right) \rightarrow \mathcal{O}$.
\end{enumerate}
\end{proposition}

We now define bounded set and diameter.

\begin{definition}
A nonempty subset $Y\subset X$ is called\ \textbf{bounded} with respect to $%
\mathcal{O}$ if there is some $\mathcal{U}\in \mathcal{O}$ such that $%
\mathcal{U}\in \rho \left( x,y\right) $ for all $x,y\in Y$.
\end{definition}

\begin{definition}
Let $Y\subset X$ be a nonempty set. The \textbf{diameter} of $Y$ is the set $%
\mathrm{D}\left( Y\right) \in \mathcal{P}\left( \mathcal{O}\right) $ defined
as
\begin{equation*}
\mathrm{D}\left( Y\right) =\bigcap\limits_{x,y\in Y}\rho \left( x,y\right) .
\end{equation*}
\end{definition}

If $Y\subset X$ is a bounded set then $\mathrm{D}\left( Y\right) \neq
\emptyset $. It is easily seen that $\mathcal{U}\in \mathrm{D}\left(
Y\right) $ if and only if $Y$ is bounded by $\mathcal{U}$. The following
properties of diameter are proved in \cite[Proposition 4]{RichardA}.

\begin{proposition}
\label{P9}

\begin{enumerate}
\item $\rho \left( x,y\right) \prec \mathrm{D}\left( A\right) $ for all $%
x,y\in A$.

\item $\mathrm{D}\left( A\right) \prec \mathrm{D}\left( B\right) $ if $%
A\subset B$.

\item $\mathrm{D}\left( A\right) \prec \mathrm{D}\left( \mathrm{cls}\left(
A\right) \right) \prec 2\mathrm{D}\left( A\right) .$
\end{enumerate}
\end{proposition}

We finally define measure of noncompactness.

\begin{definition}
Let $Y\subset X$ be a nonempty set. The \textbf{star measure of
noncompactness} of $Y$ is the set $\alpha \left( Y\right) \in \mathcal{P}%
\left( \mathcal{O}\right) $ defined as
\begin{equation*}
\alpha \left( Y\right) =\left\{ \mathcal{U}\in \mathcal{O}:Y\text{ admits a
finite cover }Y\subset \bigcup\limits_{i=1}^{n}\mathrm{St}\left[ x_{i},%
\mathcal{U}\right] \right\} ;
\end{equation*}%
the \textbf{Kuratowski measure of noncompactness} of $Y$ is the set $\gamma
\left( Y\right) \in \mathcal{P}\left( \mathcal{O}\right) $ defined as
\begin{equation*}
\gamma \left( Y\right) =\left\{ \mathcal{U}\in \mathcal{O}:Y\text{ admits a
finite cover }Y\subset \bigcup\limits_{i=1}^{n}X_{i}\text{ with }\mathcal{U}%
\in \mathrm{D}\left( X_{i}\right) \right\} .
\end{equation*}
\end{definition}

If $Y\subset X$ is a bounded set then both the sets $\alpha \left( Y\right) $
and $\gamma \left( Y\right) $ are nonempty. The following properties of
measure of noncompactness are proved in \cite[Proposition 10]{RichardA}.

\begin{proposition}
\label{P19}

\begin{enumerate}
\item $\alpha \left( Y\right) \prec \gamma \left( Y\right) \prec 1\alpha
\left( Y\right) .$

\item $\alpha \left( Y\right) \prec \mathrm{D}\left( Y\right) .$

\item $\alpha \left( Y\right) \prec \alpha \left( Z\right) $ if $Y\subset Z$.

\item $\alpha \left( Y\cup Z\right) =\alpha \left( Y\right) \cap \alpha
\left( Y\right) .$

\item $\alpha \left( Y\right) \prec \alpha \left( \mathrm{cls}\left(
Y\right) \right) \prec 1\alpha \left( Y\right) .$
\end{enumerate}
\end{proposition}

Recall that a net $\left( x_{\lambda }\right) _{\lambda \in \Lambda }$ in $X$
is $\mathcal{O}$-Cauchy if for each $\mathcal{U}\in \mathcal{O}$ there is
some $\lambda _{0}\in \Lambda $ such that such that $\mathcal{U}\in \rho
\left( x_{\lambda _{1}},x_{\lambda _{2}}\right) $ whenever $\lambda
_{1},\lambda _{2}\geqslant \lambda _{0}$. If every Cauchy net in the
admissible space $X$ converges then $X$ is called a \textbf{complete
admissible space}. If $\mathrm{cls}\left( Y\right) $ is compact then $\alpha
\left( Y\right) =\mathcal{O}$. The converse holds if $X$ is a complete
admissible space (see \cite[Proposition 11]{RichardA}). The following
theorem is proved in \cite{RichardA}.

\begin{theorem}[Cantor--Kuratowski theorem]
\label{TK} The admissible space $X$ is complete if and only if every
decreasing net $\left( F_{\lambda }\right) $ of nonempty bounded closed sets
of $X$, with $\gamma \left( F_{\lambda }\right) \rightarrow \mathcal{O}$,
has nonempty compact intersection.
\end{theorem}

\section{\label{SectionHyperspace}The hyperspace}

In this section we define hyperspace by means of a binary function of the
admissible space. Throughout, there is a fixed completely regular space $X$
endowed with an admissible family of open coverings $\mathcal{O}$.

For a given point $x\in X$ and a subset $A\subset X$, we define the set $%
\rho \left( x,A\right) \in \mathcal{P}\left( \mathcal{O}\right) $ by%
\begin{equation*}
\rho \left( x,A\right) =\bigcup\limits_{a\in A}\rho \left( x,a\right) .
\end{equation*}%
In Propositions \ref{P2} and \ref{PKS} below, we present some relevant
properties of $\rho \left( x,A\right) $.

\begin{proposition}
\label{P2}For a given point $x\in X$ and subsets $A,B\subset X$, the
following properties hold:

\begin{enumerate}
\item $\rho \left( x,A\right) \prec \rho \left( x,a\right) $ for all $a\in A$%
.

\item If $B\supset A$ then $\rho \left( x,B\right) \prec \rho \left(
x,A\right) $.

\item $\rho \left( x,A\right) =\mathcal{O}$ if and only if $x\in \mathrm{cls}%
\left( A\right) $.

\item $\rho \left( x,\mathrm{cls}\left( A\right) \right) =\rho \left(
x,A\right) .$
\end{enumerate}
\end{proposition}

\begin{proof}
Items $\left( 1\right) $ and $\left( 2\right) $ follow immediately by
definition. For item $\left( 3\right) $, note that $\rho \left( x,A\right) =%
\mathcal{O}$ if and only if $A\cap \mathrm{St}\left[ x,\mathcal{U}\right]
\neq \emptyset $ for all $\mathcal{U}\in \mathcal{O}$. Since the collection $%
\left\{ \mathrm{St}\left[ x,\mathcal{U}\right] :\mathcal{U}\in \mathcal{O}%
\right\} $ is a neighborhood base at $x$, it follows that $\rho \left(
x,A\right) =\mathcal{O}$ if and only if $x\in \mathrm{cls}\left( A\right) $.
We now prove item $\left( 4\right) $. If $\mathcal{U}\in \rho \left( x,%
\mathrm{cls}\left( A\right) \right) $ then $\mathcal{U}\in \rho \left(
x,y\right) $ for some $y\in \mathrm{cls}\left( A\right) $. Hence $y\in
\mathrm{St}\left[ x,\mathcal{U}\right] \cap \mathrm{cls}\left( A\right) $,
and therefore $\mathrm{St}\left[ x,\mathcal{U}\right] \cap A\neq \emptyset $
because $\mathrm{St}\left[ x,\mathcal{U}\right] $ is open. It follows that $%
\mathcal{U}\in \rho \left( x,a\right) $ for some $a\in A$. Thus $\mathcal{U}%
\in \rho \left( x,A\right) $ and we have the inclusion $\rho \left( x,%
\mathrm{cls}\left( A\right) \right) \subset \rho \left( x,A\right) $. The
inclusion $\rho \left( x,A\right) \subset \rho \left( x,\mathrm{cls}\left(
A\right) \right) $ is clear.
\end{proof}

\begin{proposition}
\label{PKS} Let $K$ be a compact subset of $X$ and $(x_{\lambda })_{\lambda
\in \Lambda }$ a net in $X$. If $\rho \left( x_{\lambda },K\right)
\rightarrow \mathcal{O}$ then $x_{\lambda }$ admits a convergent subnet $%
x_{\lambda _{\mu }}\rightarrow y$, with $y\in \mathrm{cls}\left( K\right) $.
\end{proposition}

\begin{proof}
For a given $\mathcal{U}\in \mathcal{O}$, there is $\lambda _{\mathcal{U}%
}\in \Lambda $ such that $\mathcal{U}\in \rho \left( x_{\lambda },K\right) $
whenever $\lambda \geq \lambda _{\mathcal{U}}$. Hence, for every $\lambda
\geq \lambda _{\mathcal{U}}$, we can take $z_{\left( \lambda ,\mathcal{U}%
\right) }\in K$ such that $\mathcal{U}\in \rho \left( x_{\lambda },z_{\left(
\lambda ,\mathcal{U}\right) }\right) $. Define the set
\begin{equation*}
\Gamma =\left\{ \left( \lambda ,\mathcal{U}\right) :\mathcal{U}\in \rho
\left( x_{\lambda },z_{\left( \lambda ,\mathcal{U}\right) }\right) \right\}
\end{equation*}%
directed by $\left( \lambda _{1},\mathcal{U}_{1}\right) \geq \left( \lambda
_{2},\mathcal{U}_{2}\right) $ if and only if $\lambda _{1}\geq \lambda _{2}$
and $\mathcal{U}_{1}\leqslant \mathcal{U}_{2}$. For each $\left( \lambda ,%
\mathcal{U}\right) \in \Gamma $, we define $x_{\left( \lambda ,\mathcal{U}%
\right) }=x_{\lambda }$. By compactness of $K$, we may assume that $%
z_{\left( \lambda ,\mathcal{U}\right) }\rightarrow z$ for some $z\in \mathrm{%
cls}\left( K\right) $. Now, for a given $\mathcal{U}\in \mathcal{O}$, take $%
\mathcal{U}^{\prime }\in \mathcal{O}$ with $\mathcal{U}^{\prime }\leqslant
\frac{1}{2}\mathcal{U}$. As $\rho \left( z_{\left( \lambda ,\mathcal{U}%
\right) },z\right) \rightarrow \mathcal{O}$, there is $\left( \lambda _{0},%
\mathcal{U}_{0}\right) $ such that $\mathcal{U}^{\prime }\in \rho \left(
z_{\left( \lambda ,\mathcal{V}\right) },z\right) $ whenever $\left( \lambda ,%
\mathcal{V}\right) \geq \left( \lambda _{0},\mathcal{U}_{0}\right) $. Choose
$\mathcal{U}_{0}^{\prime }\in \mathcal{O}$ such that $\mathcal{U}%
_{0}^{\prime }\leqslant \mathcal{U}^{\prime }$ and $\mathcal{U}_{0}^{\prime
}\leqslant \mathcal{U}_{0}$. If $\left( \lambda ,\mathcal{V}\right) \geq
\left( \lambda _{0},\mathcal{U}_{0}^{\prime }\right) $ then $\mathcal{U}%
^{\prime }\in \rho \left( x_{\lambda },z_{\left( \lambda ,\mathcal{V}\right)
}\right) $, by hereditariness since $\mathcal{V}\in \rho \left( x_{\lambda
},z_{\left( \lambda ,\mathcal{V}\right) }\right) $ and $\mathcal{U}^{\prime
}\in \rho \left( z_{\left( \lambda ,\mathcal{V}\right) },z\right) $, as $%
\left( \lambda ,\mathcal{V}\right) \geq \left( \lambda _{0},\mathcal{U}%
_{0}\right) $. It follows that%
\begin{equation*}
\rho \left( x_{\left( \lambda ,\mathcal{U}\right) },z\right) \prec 1(\rho
(x_{\lambda },z_{\left( \lambda ,\mathcal{V}\right) })\cap \rho (z_{\left(
\lambda ,\mathcal{V}\right) },z))\prec 1\{\mathcal{U}^{\prime }\}\prec \{%
\mathcal{U}\}
\end{equation*}%
for all $\left( \lambda ,\mathcal{V}\right) \geq \left( \lambda _{0},%
\mathcal{U}_{0}^{\prime }\right) $. Thus the subnet $\left( x_{\left(
\lambda ,\mathcal{U}\right) }\right) _{\left( \lambda ,\mathcal{U}\right)
\in \Gamma }$ of $\left( x_{\lambda }\right) _{\lambda \in \Lambda }$
converges to $z\in \mathrm{cls}\left( K\right) $.
\end{proof}

We now extend $\rho $ to an operation of sets.

\begin{definition}
For two closed subsets $A,B\subset X$, we define the collection $\rho
_{A}\left( B\right) \in \mathcal{P}\left( \mathcal{O}\right) $ by
\begin{equation*}
\rho _{A}\left( B\right) =\bigcap\limits_{b\in B}\rho \left( b,A\right)
=\bigcap\limits_{b\in B}\bigcup\limits_{a\in A}\rho \left( b,a\right)
\end{equation*}%
and the collection $\rho _{H}\left( A,B\right) \in \mathcal{P}\left(
\mathcal{O}\right) $ by
\begin{equation*}
\rho _{H}\left( A,B\right) =\rho _{A}\left( B\right) \cap \rho _{B}\left(
A\right) =\left\{ \bigcap\limits_{b\in B}\bigcup\limits_{a\in A}\rho \left(
a,b\right) \right\} \cap \left\{ \bigcap\limits_{a\in A}\bigcup\limits_{b\in
B}\rho \left( a,b\right) \right\} .
\end{equation*}
\end{definition}

Note that $\rho _{H}\left( A,B\right) $ is a symmetric relation, although $%
\rho _{A}\left( B\right) $ not. By Proposition \ref{P2}, $\rho _{H}\left(
A,B\right) =\rho _{H}\left( \mathrm{cls}\left( A\right) ,\mathrm{cls}\left(
B\right) \right) $ for all subsets $A,B\subset X$. Thus we may consider only
closed sets in working with the function $\rho _{H}$.

From now on, the collection of all nonempty closed subsets of $X$ will be
denoted by $\mathcal{H}\left( X\right) $. In the following we present some
properties of $\rho _{H}$ on $\mathcal{H}\left( X\right) $.

\begin{proposition}
\label{P3}The binary function $\rho _{H}:\mathcal{H}\left( X\right) \times
\mathcal{H}\left( X\right) \rightarrow \mathcal{P}\left( \mathcal{O}\right) $
satisfies the following properties:

\begin{enumerate}
\item $\rho _{H}\left( A,B\right) =\rho _{H}\left( B,A\right) $ for all $%
A,B\in \mathcal{H}\left( X\right) $.

\item $\rho _{H}\left( \left\{ x\right\} ,\left\{ y\right\} \right) =\rho
\left( x,y\right) $ for all $x,y\in X$.

\item $\rho _{H}\left( A,B\right) =\mathcal{O}$ if and only if $A=B$.

\item $\rho _{H}\left( A,C\right) \prec 1\left( \rho _{H}\left( A,B\right)
\cap \rho _{H}\left( B,C\right) \right) $ for all $A,B,C\in \mathcal{H}%
\left( X\right) $.

\item $\rho _{H}\left( A\cup B,C\cup D\right) \prec \rho _{H}\left(
A,C\right) \cap \rho _{H}\left( B,D\right) $ for all $A,B,C,D\in \mathcal{H}%
\left( X\right) $.

\item $\mathcal{U}\in \rho _{H}\left( A,B\right) $ if and only if $A\subset
\mathrm{St}\left[ B,\mathcal{U}\right] $ and $B\subset \mathrm{St}\left[ A,%
\mathcal{U}\right] $.
\end{enumerate}
\end{proposition}

\begin{proof}
Items $\left( 1\right) $ and $\left( 2\right) $ are immediate from the
definition. For item $\left( 3\right) $, note that $\rho _{A}\left( B\right)
=\mathcal{O}$ if and only if $B\subset \mathrm{cls}\left( A\right) =A$, by
Proposition \ref{P2}. Analogously, $\rho _{B}\left( A\right) =\mathcal{O}$
if and only if $A\subset B$. Hence, $\rho _{H}\left( A,B\right) =\mathcal{O}$
if and only if $\rho _{A}\left( B\right) =\rho _{B}\left( A\right) =\mathcal{%
O}$ if and only if $A=B$. For item $\left( 4\right) $, suppose that $%
\mathcal{U}\in 1\left( \rho _{H}\left( A,B\right) \cap \rho _{H}\left(
B,C\right) \right) $. Then there is $\mathcal{V}\in \rho _{H}\left(
A,B\right) \cap \rho _{H}\left( B,C\right) $ such that $\mathcal{V}\leqslant
\frac{1}{2}\mathcal{U}$. It follows that
\begin{eqnarray*}
\mathcal{V} &\in &\rho _{A}\left( B\right) \cap \rho _{B}\left( A\right)
\cap \rho _{B}\left( C\right) \cap \rho _{C}\left( B\right) \\
&=&\left\{ \bigcap\limits_{b\in B}\bigcup\limits_{a\in A}\rho \left(
a,b\right) \right\} \cap \left\{ \bigcap\limits_{a\in A}\bigcup\limits_{b\in
B}\rho \left( a,b\right) \right\} \cap \left\{ \bigcap\limits_{c\in
C}\bigcup\limits_{b\in B}\rho \left( b,c\right) \right\} \cap \left\{
\bigcap\limits_{b\in B}\bigcup\limits_{c\in C}\rho \left( b,c\right)
\right\} .
\end{eqnarray*}%
For a given $c\in C$, there is $b\in B$ such that $\mathcal{V}\in \rho
\left( b,c\right) $, because $\mathcal{V}\in \rho _{B}\left( C\right) $. For
this $b$, there is $a\in A$ such that $\mathcal{V}\in \rho \left( a,b\right)
$, since $\mathcal{V}\in \rho _{A}\left( B\right) $. Hence $a,c\in \mathrm{St%
}\left[ b,\mathcal{V}\right] $, which implies $a\in \mathrm{St}\left[ c,%
\mathcal{U}\right] $, as $\mathcal{V}\leqslant \frac{1}{2}\mathcal{U}$. Thus
$\mathcal{U}\in \rho \left( a,c\right) $. Since $c\in C$ is arbitrary, it
follows that $\mathcal{U}\in \bigcap\limits_{c\in C}\rho \left( c,A\right)
=\rho _{A}\left( C\right) $. Similarly, by using $\mathcal{V}\in \rho
_{B}\left( A\right) \cap \rho _{C}\left( B\right) $, we obtain $\mathcal{U}%
\in \bigcap\limits_{a\in A}\rho \left( a,C\right) =\rho _{A}\left( C\right) $%
. Therefore $1\left( \rho _{H}\left( A,B\right) \cap \rho _{H}\left(
B,C\right) \right) \subset \rho _{A}\left( C\right) \cap \rho _{C}\left(
A\right) =\rho _{H}\left( A,C\right) $. We now show the item $\left(
5\right) $. Suppose that $\mathcal{U}\in \rho _{H}\left( A,C\right) \cap
\rho _{H}\left( B,D\right) $ and let $x\in A\cup B$ and $y\in C\cup D$. In
the case $x\in A$, we have
\begin{equation*}
\mathcal{U}\in \rho _{C}\left( A\right) \subset \rho \left( x,C\right)
\subset \rho \left( x,C\cup D\right) .
\end{equation*}%
If $x\in B$, we have
\begin{equation*}
\mathcal{U}\in \rho _{D}\left( B\right) \subset \rho \left( x,D\right)
\subset \rho \left( x,C\cup D\right) .
\end{equation*}%
Hence $\mathcal{U}\in \rho _{C\cup D}\left( A\cup B\right) $. If $y\in C$ or
$y\in D$, we have respectively
\begin{equation*}
\mathcal{U}\in \rho _{A}\left( C\right) \subset \rho \left( y,A\right)
\subset \rho \left( y,A\cup B\right)
\end{equation*}%
or
\begin{equation*}
\mathcal{U}\in \rho _{B}\left( D\right) \subset \rho \left( y,B\right)
\subset \rho \left( y,A\cup B\right) .
\end{equation*}%
Hence $\mathcal{U}\in \rho _{A\cup B}\left( C\cup D\right) $, and therefore $%
\mathcal{U}\in \rho _{H}\left( A\cup B,C\cup D\right) $. It follows that $%
\rho _{H}\left( A,C\right) \cap \rho _{H}\left( B,D\right) \subset \rho
_{H}\left( A\cup B,C\cup D\right) $. Finally, we show the item $\left(
6\right) $. We have $\mathcal{U}\in \rho _{H}\left( A,B\right) $ if and only
if $\mathcal{U}\in \rho \left( a,B\right) $, for every $a\in A$, and $%
\mathcal{U}\in \rho \left( b,A\right) $, for every $b\in B$, which means
that $A\subset \mathrm{St}\left[ B,\mathcal{U}\right] $ and $B\subset
\mathrm{St}\left[ A,\mathcal{U}\right] $.
\end{proof}

The function $\rho _{H}$ also relates to the notion of diameter. In fact, by
estimating the diameter of a reunion of sets, the \textquotedblleft
distance\textquotedblright\ between the sets should be regarded, as the
following.

\begin{proposition}
Let $A,B\subset X$ be nonempty subsets. The following properties hold:

\begin{enumerate}
\item $\mathrm{D}\left( A\cup B\right) \prec 1\left( \mathrm{D}\left(
A\right) \cap \mathrm{D}\left( B\right) \cap \rho _{A}\left( B\right)
\right) .$

\item $\mathrm{D}\left( A\cup B\right) \prec 1\left( \mathrm{D}\left(
A\right) \cap \mathrm{D}\left( B\right) \cap \rho _{B}\left( A\right)
\right) .$

\item $\mathrm{D}\left( A\cup B\right) \prec 1\left( \mathrm{D}\left(
A\right) \cap \mathrm{D}\left( B\right) \cap \rho _{H}\left( A,B\right)
\right) .$
\end{enumerate}
\end{proposition}

\begin{proof}
For item $\left( 1\right) $ we suppose the nontrivial case $1\left( \mathrm{D%
}\left( A\right) \cap \mathrm{D}\left( B\right) \cap \rho _{H}\left(
A,B\right) \right) \neq \emptyset $. If $\mathcal{U}\in 1\left( \mathrm{D}%
\left( A\right) \cap \mathrm{D}\left( B\right) \cap \rho _{H}\left(
A,B\right) \right) $ then there is $\mathcal{V}\in \mathrm{D}\left( A\right)
\cap \mathrm{D}\left( B\right) \cap \rho _{H}\left( A,B\right) $ such that $%
\mathcal{V}\leqslant \frac{1}{2}\mathcal{U}$. Let $x,y\in A\cup B$. If $%
x,y\in A$, we have $\mathcal{V}\in \rho \left( x,y\right) $, and then $%
\mathcal{U}\in \rho \left( x,y\right) $ because $\rho \left( x,y\right) $ is
upward hereditary. By the same reason, $\mathcal{U}\in \rho \left(
x,y\right) $ if $x,y\in B$. Suppose $x\in A$ and $y\in B$. Since $\mathcal{V}%
\in \rho _{A}\left( B\right) =\bigcap\limits_{b\in B}\bigcup\limits_{a\in
A}\rho \left( a,b\right) $, there is $a\in A$ such that $\mathcal{V}\in \rho
\left( a,y\right) $. As $\mathcal{V}\in \mathrm{D}\left( A\right) $, we have
$\mathcal{V}\in \rho \left( x,a\right) $. Hence $\mathcal{V}\in \rho \left(
x,a\right) \cap \rho \left( a,y\right) $ and then $\mathcal{U}\in 1\left(
\rho \left( x,a\right) \cap \rho \left( a,y\right) \right) $. Since $\rho
\left( x,y\right) \prec 1\left( \rho \left( x,a\right) \cap \rho \left(
a,y\right) \right) $, it follows that $\mathcal{U}\in \rho \left( x,y\right)
$. In any case we have $\mathcal{U}\in \rho \left( x,y\right) $ for
arbitraries $x,y\in A\cup B$. Therefore $\mathcal{U}\in \mathrm{D}\left(
A\cup B\right) $. Item $\left( 2\right) $ can be analogously proved and item
$\left( 3\right) $ is a straightforward consequence of item $\left( 1\right)
$ together with item $\left( 2\right) $.
\end{proof}

In order to provide a uniformity on $\mathcal{H}\left( X\right) $, we
construct a base for diagonal uniformity by means of the function $\rho _{H}$%
. For each $A\in \mathcal{H}\left( X\right) $ and $\mathcal{U}\in \mathcal{O}
$, we define the set $\mathrm{B}_{H}\left( A,\mathcal{U}\right) $ in $%
\mathcal{H}\left( X\right) $ by
\begin{eqnarray*}
\mathrm{B}_{H}\left( A,\mathcal{U}\right) &=&\left\{ B\in \mathcal{H}\left(
X\right) :\mathcal{U}\in \rho _{H}\left( A,B\right) \right\} \\
&=&\left\{ B\in \mathcal{H}\left( X\right) :A\subset \mathrm{St}\left[ B,%
\mathcal{U}\right] \text{ and }B\subset \mathrm{St}\left[ A,\mathcal{U}%
\right] \right\} .
\end{eqnarray*}

We now construct a diagonal uniformity on $\mathcal{H}\left( X\right) $ by
means of the sets $\mathrm{B}_{H}\left( A,\mathcal{U}\right) $. It should be
remembered that for two surroundings $D,E\subset \mathcal{H}\left( X\right)
\times \mathcal{H}\left( X\right) $, one has the composition
\begin{equation*}
D\circ E=\left\{ \left( A,B\right) \in \mathcal{H}\left( X\right) \times
\mathcal{H}\left( X\right) :\left( A,C\right) \in E\text{ and }\left(
C,B\right) \in D\text{ for some }C\in \mathcal{H}\left( X\right) \right\}
\end{equation*}%
and the inverse
\begin{equation*}
D^{-1}=\left\{ \left( A,B\right) \in \mathcal{H}\left( X\right) \times
\mathcal{H}\left( X\right) :\left( B,A\right) \in D\right\} .
\end{equation*}

\begin{theorem}
\label{T4} The collection of the sets $D_{\mathcal{U}}=\bigcup \left\{
\mathrm{B}_{H}\left( A,\mathcal{U}\right) \times \mathrm{B}_{H}\left( A,%
\mathcal{U}\right) :A\in \mathcal{H}\left( X\right) \right\} $, for $%
\mathcal{U}\in \mathcal{O}$, is a base for a diagonal uniformity $\mathfrak{D%
}_{\mathcal{H}}$ on $\mathcal{H}\left( X\right) $.
\end{theorem}

\begin{proof}
Let $\Delta \subset \mathcal{H}\left( X\right) \times \mathcal{H}\left(
X\right) $ be the diagonal of $\mathcal{H}\left( X\right) $, that is, $%
\Delta =\left\{ \left( A,A\right) :A\in \mathcal{H}\left( X\right) \right\} $%
. It is easily seen that $\Delta \subset D_{\mathcal{U}}$ for all $\mathcal{U%
}\in \mathcal{O}$. For $\mathcal{U},\mathcal{V}\in \mathcal{O}$, take $%
\mathcal{W}\in \mathcal{O}$ such that $\mathcal{W}\leqslant \frac{1}{2}%
\mathcal{U}$ and $\mathcal{W}\leqslant \frac{1}{2}\mathcal{V}$. If $\left(
A,B\right) \in D_{\mathcal{W}}$ then there is $C\in \mathcal{H}\left(
X\right) $ such that $\left( A,B\right) \in \mathrm{B}_{H}\left( C,\mathcal{W%
}\right) \times \mathrm{B}_{H}\left( C,\mathcal{W}\right) $. Hence
\begin{eqnarray*}
A &\subset &\mathrm{St}\left[ C,\mathcal{W}\right] \text{ and }C\subset
\mathrm{St}\left[ A,\mathcal{W}\right] , \\
B &\subset &\mathrm{St}\left[ C,\mathcal{W}\right] \text{ and }C\subset
\mathrm{St}\left[ B,\mathcal{W}\right] .
\end{eqnarray*}%
Thus $A\subset \mathrm{St}\left[ C,\mathcal{W}\right] $ and $C\subset
\mathrm{St}\left[ B,\mathcal{W}\right] $, which implies $A\subset \mathrm{St}%
\left[ B,\mathcal{U}\right] $, since $\mathcal{W}\leqslant \frac{1}{2}%
\mathcal{U}$. On the other hand $B\subset \mathrm{St}\left[ C,\mathcal{W}%
\right] $ and $C\subset \mathrm{St}\left[ A,\mathcal{W}\right] $, which
implies $B\subset \mathrm{St}\left[ A,\mathcal{U}\right] $. Hence $A\in
\mathrm{B}_{H}\left( B,\mathcal{U}\right) $, and therefore $\left(
A,B\right) \in D_{\mathcal{U}}$. As $\mathcal{W}\leqslant \frac{1}{2}%
\mathcal{V}$, we can similarly prove that $\left( A,B\right) \in D_{\mathcal{%
V}}$. Thus $D_{\mathcal{W}}\subset D_{\mathcal{U}}\cap D_{\mathcal{V}}$.
Now, for a given surrounding $D_{\mathcal{U}}$, we should find $\mathcal{V},%
\mathcal{W}\in \mathcal{O}$ such that $D_{\mathcal{V}}\circ D_{\mathcal{V}%
}\subset D_{\mathcal{U}}$ and $D_{\mathcal{W}}^{-1}\subset D_{\mathcal{U}}$.
For indeed, since $D_{\mathcal{U}}$ is symmetric, we have $D_{\mathcal{U}%
}^{-1}=D_{\mathcal{U}}$. Take $\mathcal{V}\in \mathcal{O}$ such that $%
\mathcal{V}\leqslant \frac{1}{2}\mathcal{U}$. If $\left( A,B\right) \in D_{%
\mathcal{V}}\circ D_{\mathcal{V}}$ then there is $C\in \mathcal{H}\left(
X\right) $ such that $\left( A,C\right) ,\left( C,B\right) \in D_{\mathcal{V}%
}$. Hence there are $E,F\in \mathcal{H}\left( X\right) $ such that $\left(
A,C\right) \in \mathrm{B}_{H}\left( E,\mathcal{V}\right) \times \mathrm{B}%
_{H}\left( E,\mathcal{V}\right) $ and $\left( C,B\right) \in \mathrm{B}%
_{H}\left( F,\mathcal{V}\right) \times \mathrm{B}_{H}\left( F,\mathcal{V}%
\right) $. This means that
\begin{eqnarray*}
A &\subset &\mathrm{St}\left[ E,\mathcal{V}\right] \text{ and }E\subset
\mathrm{St}\left[ A,\mathcal{V}\right] , \\
C &\subset &\mathrm{St}\left[ E,\mathcal{V}\right] \text{ and }E\subset
\mathrm{St}\left[ C,\mathcal{V}\right] , \\
C &\subset &\mathrm{St}\left[ F,\mathcal{V}\right] \text{ and }F\subset
\mathrm{St}\left[ C,\mathcal{V}\right] , \\
B &\subset &\mathrm{St}\left[ F,\mathcal{V}\right] \text{ and }F\subset
\mathrm{St}\left[ B,\mathcal{V}\right] .
\end{eqnarray*}%
As $\mathcal{V}\leqslant \frac{1}{2}\mathcal{U}$, it follows that
\begin{eqnarray*}
A &\subset &\mathrm{St}\left[ C,\mathcal{U}\right] \text{ and }C\subset
\mathrm{St}\left[ A,\mathcal{U}\right] , \\
B &\subset &\mathrm{St}\left[ C,\mathcal{U}\right] \text{ and }C\subset
\mathrm{St}\left[ B,\mathcal{U}\right] .
\end{eqnarray*}%
Hence $\left( A,B\right) \in \mathrm{B}_{H}\left( C,\mathcal{U}\right)
\times \mathrm{B}_{H}\left( C,\mathcal{U}\right) $, and therefore $D_{%
\mathcal{V}}\circ D_{\mathcal{V}}\subset D_{\mathcal{U}}$.
\end{proof}

The resulting uniform space $\left( \mathcal{H}\left( X\right) ,\mathfrak{D}%
_{\mathcal{H}}\right) $ is called \emph{hyperspace} of $X$ with respect to
the admissible family $\mathcal{O}$.

Recall that the uniform topology on $\mathcal{H}\left( X\right) $ generated
by $\mathfrak{D}_{\mathcal{H}}$ is the topology whose neighborhood base at $%
A\in \mathcal{H}\left( X\right) $ is formed by the collection $\mathcal{N}%
_{A}=\left\{ D\left[ A\right] :D\in \mathfrak{D}_{\mathcal{H}}\right\} $
where
\begin{equation*}
D\left[ A\right] =\left\{ B\in \mathcal{H}\left( X\right) :\left( A,B\right)
\in D\right\} .
\end{equation*}%
The same topology is produced if we consider only elements $D_{\mathcal{U}}$
in the base for $\mathfrak{D}_{\mathcal{H}}$ (see \cite[Theorem 35.6]{Will}%
). The covering uniformity $\mathcal{O}_{\mathcal{H}}$ associated to the
diagonal uniformity $\mathfrak{D}_{\mathcal{H}}$ has a base of uniform
coverings of the form $\upsilon \left( D\right) =\left\{ D\left[ A\right]
:A\in \mathcal{H}\left( X\right) \right\} $ for $D\in \mathfrak{D}_{\mathcal{%
H}}$.

\begin{theorem}
\label{T3}For any $A\in \mathcal{H}\left( X\right) $, one has the inclusions:

\begin{enumerate}
\item $D_{\mathcal{V}}\left[ A\right] \subset \mathrm{B}_{H}\left( A,%
\mathcal{U}\right) \subset D_{\mathcal{U}}\left[ A\right] $ whenever $%
\mathcal{V}\leqslant \frac{1}{2}\mathcal{U}$.

\item $\mathrm{St}\left[ A,\upsilon \left( D_{\mathcal{V}}\right) \right]
\subset \mathrm{B}_{H}\left( A,\mathcal{U}\right) \subset \mathrm{St}\left[
A,\upsilon \left( D_{\mathcal{U}}\right) \right] $ whenever $\mathcal{V}%
\leqslant \frac{1}{2^{2}}\mathcal{U}$.
\end{enumerate}
\end{theorem}

\begin{proof}
$\left( 1\right) $ If $B\in D_{\mathcal{V}}\left[ A\right] $ then $\left(
A,B\right) \in D_{\mathcal{V}}$, hence there is $C\in \mathcal{H}\left(
X\right) $ such that $\left( A,B\right) \in \mathrm{B}_{H}\left( C,\mathcal{V%
}\right) \times \mathrm{B}_{H}\left( C,\mathcal{V}\right) $. As $\mathcal{V}%
\leqslant \frac{1}{2}\mathcal{U}$, it follows that $B\in \mathrm{B}%
_{H}\left( A,\mathcal{U}\right) $. The inclusion $\mathrm{B}_{H}\left( A,%
\mathcal{U}\right) \subset D_{\mathcal{U}}\left[ A\right] $ is obvious.

$\left( 2\right) $ Let $\mathcal{W}\in \mathcal{O}$ such that $\mathcal{V}%
\leqslant \frac{1}{2}\mathcal{W}$ and $\mathcal{W}\leqslant \frac{1}{2}%
\mathcal{U}$. If $B\in \mathrm{St}\left[ A,\upsilon \left( D_{\mathcal{V}%
}\right) \right] $ then $A,B\in D_{\mathcal{V}}\left[ C\right] $ for some $%
C\in \mathcal{H}\left( X\right) $. By item $\left( 1\right) $, we have $D_{%
\mathcal{V}}\left[ C\right] \subset \mathrm{B}_{H}\left( C,\mathcal{W}%
\right) $, hence $A,B\in \mathrm{B}_{H}\left( C,\mathcal{W}\right) $. As $%
\mathcal{W}\leqslant \frac{1}{2}\mathcal{U}$, it follows that $B\in \mathrm{B%
}_{H}\left( A,\mathcal{U}\right) $. Thus $\mathrm{St}\left[ A,\upsilon
\left( D_{\mathcal{V}}\right) \right] \subset \mathrm{B}_{H}\left( A,%
\mathcal{U}\right) $. Now, since $\mathrm{B}_{H}\left( A,\mathcal{U}\right)
\subset D_{\mathcal{U}}\left[ A\right] $, the inclusion $\mathrm{B}%
_{H}\left( A,\mathcal{U}\right) \subset \mathrm{St}\left[ A,\upsilon \left(
D_{\mathcal{U}}\right) \right] $ is clear.
\end{proof}

The following results are immediate consequences of Proposition \ref{P3},
item 3, and Theorem \ref{T3}.

\begin{corollary}
\label{EC} Let $\left( F_{\lambda }\right) $ be a net in $\mathcal{H}(X)$
and $F\in \mathcal{H}\left( X\right) $. Then $F_{\lambda }\rightarrow F$ if
and only if $\rho _{H}\left( F_{\lambda },F\right) \rightarrow \mathcal{O}$.
\end{corollary}

\begin{corollary}
The uniform topology in $\mathcal{H}\left( X\right) $ is Hausdorff.
\end{corollary}

Note that the convergence in the uniform topology generalizes the notion of
Hausdorff convergence in metric spaces.

We now discuss the case of compact admissible space. Assume that $X$ is a
compact Hausdorff space and let $\mathcal{O}_{f}$ be the admissible family
of all finite open coverings of $X$. Then the hyperspace $\mathcal{H}\left(
X\right) $ is the set of all nonempty compact subsets of $X$. We shall prove
that the uniform topology of $\mathcal{H}\left( X\right) $ coincides with
the Hausdorff topology. For a given finite collection $\mathcal{C}=\left\{
U_{1},...,U_{n}\right\} $ of open sets in $X$, we define the set $%
\left\langle \mathcal{C}\right\rangle \subset \mathcal{H}\left( X\right) $
by
\begin{equation*}
\left\langle \mathcal{C}\right\rangle =\left\{ A\in \mathcal{H}\left(
X\right) :A\subset \bigcup\limits_{i=1}^{n}U_{i}\text{ and }A\cap U_{i}\neq
\emptyset \text{ for every }i=1,...,n\right\} .
\end{equation*}%
For $A\in \mathcal{H}\left( X\right) $ and $\mathcal{U}\in \mathcal{O}_{f}$,
we define the collection
\begin{equation*}
\left[ A,\mathcal{U}\right] =\left\{ U\in \mathcal{U}:A\cap U\neq \emptyset
\right\} .
\end{equation*}%
It is well-known that the set $\mathcal{B}_{\mathcal{H}}=\left\{
\left\langle \left[ A,\mathcal{U}\right] \right\rangle :A\in \mathcal{H}%
\left( X\right) ,\mathcal{U}\in \mathcal{O}_{f}\right\} $ is a base for the
compact Hausdorff topology of $\mathcal{H}\left( X\right) $ (\cite{Michael}).

\begin{theorem}
\label{T5}The uniform topology on $\mathcal{H}\left( X\right) $, generated
by the diagonal uniformity $\mathcal{D}_{\mathcal{H}}$, coincides with the
Hausdorff topology.
\end{theorem}

\begin{proof}
According to Theorem \ref{T3}, it is enough to show that for each $A\in
\mathcal{H}\left( X\right) $ and $\mathcal{U}\in \mathcal{O}_{f}$ there is $%
\mathcal{V}\in \mathcal{O}_{f}$ such that $\mathrm{B}_{H}\left( A,\mathcal{V}%
\right) \subset \left\langle \left[ A,\mathcal{U}\right] \right\rangle
\subset \mathrm{B}_{H}\left( A,\mathcal{U}\right) $. Firstly, note that $%
\mathrm{St}\left[ A,\mathcal{U}\right] =\bigcup\limits_{U\in \left[ A,%
\mathcal{U}\right] }U$. If $B\in \left\langle \left[ A,\mathcal{U}\right]
\right\rangle $ then $B\subset \bigcup\limits_{U\in \left[ A,\mathcal{U}%
\right] }U$ and $B\cap U\neq \emptyset $ for every $U\in \left[ A,\mathcal{U}%
\right] $. Hence $B\subset \mathrm{St}\left[ A,\mathcal{U}\right] $ and $%
\left[ A,\mathcal{U}\right] \subset \left[ B,\mathcal{U}\right] $. The
second inclusion implies $A\subset \bigcup\limits_{U\in \left[ B,\mathcal{U}%
\right] }U=\mathrm{St}\left[ B,\mathcal{U}\right] $. Hence $B\in \mathrm{B}%
_{H}\left( A,\mathcal{U}\right) $, and therefore $\left\langle \left[ A,%
\mathcal{U}\right] \right\rangle \subset \mathrm{B}_{H}\left( A,\mathcal{U}%
\right) $. Now, let $\left[ A,\mathcal{U}\right] =\left\{
U_{1},...,U_{n}\right\} $. For each $i=1,...,n$, choose $x_{i}\in A\cap
U_{i} $ and define $\mathcal{V}_{i}\in \mathcal{O}_{f}$ by%
\begin{equation*}
\mathcal{V}_{i}=\left\{ U_{1}\setminus \left\{ x_{i}\right\}
,...,U_{i},...,U_{n}\setminus \left\{ x_{i}\right\} ,X\setminus A\right\} .
\end{equation*}%
Take $\mathcal{V}\in \mathcal{O}_{f}$ such that $\mathcal{V}\leqslant
\mathcal{V}_{i}$ for every $i=1,...,n$. We claim that $\mathrm{B}_{H}\left(
A,\mathcal{V}\right) \subset \left\langle \left[ A,\mathcal{U}\right]
\right\rangle $. For indeed, if $B\in \mathrm{B}_{H}\left( A,\mathcal{V}%
\right) $ then $B\subset \mathrm{St}\left[ A,\mathcal{V}\right] $ and $%
A\subset \mathrm{St}\left[ B,\mathcal{V}\right] $. Since $\mathcal{V}%
\leqslant \mathcal{V}_{i}$, we have $B\subset \mathrm{St}\left[ A,\mathcal{V}%
_{i}\right] $ and $A\subset \mathrm{St}\left[ B,\mathcal{V}_{i}\right] $. As
\begin{equation*}
\left[ A,\mathcal{V}_{i}\right] \subset \left\{ U_{1}\setminus \left\{
x_{i}\right\} ,...,U_{i},...,U_{n}\setminus \left\{ x_{i}\right\} \right\}
\end{equation*}%
it follows that
\begin{equation*}
B\subset \left( U_{1}\setminus \left\{ x_{i}\right\} \right) \cup \ldots
\cup U_{i}\cup \ldots \cup \left( U_{n}\setminus \left\{ x_{i}\right\}
\right) \subset \bigcup\limits_{j=1}^{n}U_{j}.
\end{equation*}%
To conclude that $B\in \left\langle \left[ A,\mathcal{U}\right]
\right\rangle $, it remains to prove the inclusion $\left[ A,\mathcal{U}%
\right] \subset \left[ B,\mathcal{U}\right] $. Suppose by contradiction that
$B\cap U_{i}=\emptyset $ for some $U_{i}\in \left[ A,\mathcal{U}\right] $.
Then $U_{i}\notin \left[ B,\mathcal{V}_{i}\right] $. As $A\subset \mathrm{St}%
\left[ B,\mathcal{V}_{i}\right] $, it follows that $A\subset \left(
\bigcup\limits_{j=1,j\neq i}^{n}U_{j}\right) \setminus \left\{ x_{i}\right\}
$, which is a contradiction. Hence $B\cap U_{i}\neq \emptyset $, and
therefore $\left[ A,\mathcal{U}\right] \subset \left[ B,\mathcal{U}\right] $.
\end{proof}

\section{\label{HK}Hyperconvergence}

In this section we present some analogues of classical theorems involving
set convergence on hyperspace. We define the notion of Kuratowski
hyperconvergence and show that it coincides with the Hausdorff convergence
in the compact case. Throughout, there is a fixed admissible space $X$
endowed with an admissible family of open coverings $\mathcal{O}$.

For the following, we call \emph{Hausdorff convergence} the convergence with
respect to the uniform topology on $\mathcal{H}\left( X\right) $.

\begin{proposition}
\label{b2'} Let $\left( F_{\lambda }\right) ,\left( G_{\lambda }\right) $ be
nets in $\mathcal{H}\left( X\right) $ and $F,G\in \mathcal{H}\left( X\right)
$. The following statement holds:

\begin{enumerate}
\item If $\rho _{F}\left( F_{\lambda }\right) \rightarrow \mathcal{O}$ and $%
F\subset G$ then $\rho _{G}(F_{\lambda })\rightarrow \mathcal{O}$.

\item If $G_{\lambda }\subset F_{\lambda }$, for every $\lambda $, and $\rho
_{F}(F_{\lambda })\rightarrow \mathcal{O}$ then $\rho _{F}(G_{\lambda
})\rightarrow \mathcal{O}$.

\item If $\rho _{G}(F_{\lambda })\rightarrow \mathcal{O}$ and $\rho
_{F_{\lambda }}(F)\rightarrow \mathcal{\mathcal{\mathcal{\mathcal{\mathcal{O}%
}}}}$ then $F\subset G.$

\item If If $G_{\lambda }\subset F_{\lambda }$, for every $\lambda $, $%
F_{\lambda }\rightarrow F$, and $G_{\lambda }\rightarrow G$ then $F\subset G$%
.
\end{enumerate}
\end{proposition}

\begin{proof}
Items $\left( 1\right) ,\left( 2\right) $ are obvious since $\rho _{G}\left(
F_{\lambda }\right) \prec \rho _{F}\left( F_{\lambda }\right) $ and $\rho
_{F}\left( G_{\lambda }\right) \prec \rho _{F}\left( F_{\lambda }\right) $.
For item $\left( 3\right) $, we have $\rho _{G}\left( F\right) \prec 1\left(
\rho _{F_{\lambda }}\left( F\right) \cap \rho _{F}\left( F_{\lambda }\right)
\right) $. Since $\rho _{F_{\lambda }}\left( F\right) \cap \rho _{F}\left(
F_{\lambda }\right) \rightarrow \mathcal{O}$, it follows that $1\left( \rho
_{F_{\lambda }}\left( F\right) \cap \rho _{F}\left( F_{\lambda }\right)
\right) \rightarrow \mathcal{O}$, and therefore $\rho _{G}\left( F\right) =%
\mathcal{O}$. This means that $F\subset G$. For proving item $\left(
4\right) $, let $\mathcal{U},\mathcal{V}\in \mathcal{O}$ with $\mathcal{V}%
\leqslant \frac{1}{2}\mathcal{U}$. As $\rho _{H}\left( F_{\lambda },F\right)
\rightarrow \mathcal{O}$ and $\rho _{H}\left( G_{\lambda },G\right)
\rightarrow \mathcal{O}$, there exists $\lambda _{0}$ such that $\mathcal{V}%
\in \rho _{H}\left( F_{\lambda },F\right) \cap \rho _{H}\left( G_{\lambda
},G\right) $ for all $\lambda \geq \lambda _{0}$. Hence $\mathcal{V}\in \rho
_{F_{\lambda }}\left( F\right) \cap \rho _{G}\left( G_{\lambda }\right) $
whenever $\lambda \geq \lambda _{0}$. Now, since $F_{\lambda }\subset
G_{\lambda }$, we have $\rho _{G}\left( F_{\lambda }\right) \prec \rho
_{G}\left( G_{\lambda }\right) $, and then $\mathcal{V}\in \rho _{G}\left(
F_{\lambda }\right) $ for $\lambda \geq \lambda _{0}$. Hence
\begin{equation*}
\rho _{G}(F)\prec 1(\rho _{F_{\lambda }}(F)\cap \rho _{G}(F_{\lambda
}))\prec 1\left\{ \mathcal{V}\right\} \prec \{\mathcal{U}\}.
\end{equation*}%
This means that $\rho _{G}\left( F\right) =\mathcal{O}$ and therefore $%
F\subset G$.
\end{proof}

\begin{proposition}
\label{b2''} Assume that $X$ is a Hausdorff space. Let $\left( K_{\lambda
}\right) \subset \mathcal{H}\left( X\right) $ be a net of compacts subsets
and $B\subset X$ a compact subset. If $K_{\lambda }\rightarrow K$, with $K$
compact and $K_{\lambda }\cap B\neq \emptyset $ for all $\lambda $, then $%
K\cap B\neq \emptyset $.
\end{proposition}

\begin{proof}
Let $\left( x_{\lambda }\right) $ be a net with $x_{\lambda }\in K_{\lambda
}\cap B$. Since $B$ is compact, we may assume that $x_{\lambda }\rightarrow
a $ for some $a\in B$. Now, considerer $\mathcal{U},\mathcal{V}\in \mathcal{O%
}$ with $\mathcal{V}\leqslant \frac{1}{2}\mathcal{U}$. As $\rho _{H}\left(
K_{\lambda },K\right) \rightarrow \mathcal{O}$, there exists $\lambda _{0}$
such that $\mathcal{V}\in \rho _{K}\left( K_{\lambda }\right) $ for all $%
\lambda \geq \lambda _{0}$. This step concludes that $\mathcal{V}\in \rho
_{K}\left( x_{\lambda }\right) $ whenever $\lambda \geq \lambda _{0}$. Then
we can take some point $a^{\prime }\in K$ such that $\mathcal{V}\in \rho
\left( a^{\prime },x_{\lambda }\right) $ for all $\lambda \geq \lambda _{0}$%
. Since $\rho \left( x_{\lambda },a\right) \rightarrow \mathcal{O}$, there
exists $\lambda _{1}$ such that $\mathcal{V}\in \rho \left( x_{\lambda
},a\right) $ for all $\lambda \geq \lambda _{1}$. For $\lambda \geq \lambda
_{0},\lambda _{1}$, we have
\begin{equation*}
\rho \left( a,K\right) \prec \rho (a,a^{\prime })\prec 1(\rho (a,x_{\lambda
})\cap \rho (x_{\lambda },a^{\prime }))\prec 1\{\mathcal{V}\}\prec \{%
\mathcal{U}\}.
\end{equation*}%
It follows that $\rho (a,K)=\mathcal{O}$, which implies $a\in K$. Therefore $%
a\in K\cap B$.
\end{proof}

We now define Kuratowski convergence.

\begin{definition}
Let $\left( F_{\lambda }\right) $ be a net in $\mathcal{H}\left( X\right) $.
The sets
\begin{eqnarray*}
\mathcal{LS}\left( F_{\lambda }\right) &=&\left\{ x\in X:\text{for every
open neighborhood }U\text{ of }x\text{, }U\cap F_{\lambda }\neq \emptyset
\hspace{0.1cm}\text{frequently}\right\} , \\
\mathcal{LI}\left( F_{\lambda }\right) &=&\left\{ x\in X:\text{for every
open neighborhood }U\text{ of }x\text{, }U\cap F_{\lambda }\neq \emptyset
\hspace{0.1cm}\text{residually}\right\}
\end{eqnarray*}%
are called respectively \textbf{upper limit} and \textbf{lower limit} of $%
\left( F_{\lambda }\right) $. We say that $\left( F_{\lambda }\right) $ is
\textbf{Kuratowski convergent} to $F\in \mathcal{H}(X)$ ($F_{\lambda }%
\overset{k}{\rightarrow }F$ for short) if and only if $F=\mathcal{LS}\left(
F_{\lambda }\right) =\mathcal{LI}\left( F_{\lambda }\right) $.
\end{definition}

Note that $\mathcal{LI}\left( F_{\lambda }\right) \subset \mathcal{LS}\left(
F_{\lambda }\right) $. Then, to verify the Kuratowski convergence, it is
enough to prove the inclusions $\mathcal{LS}\left( F_{\lambda }\right)
\subset F\subset \mathcal{LI}\left( F_{\lambda }\right) $.

\begin{theorem}
\label{PK} Let $(F_{\lambda })$ be a net in $\mathcal{H}\left( X\right) $.

\begin{enumerate}
\item If $\rho _{F}(F_{\lambda })\rightarrow \mathcal{O}$ then $\mathcal{LS}%
\left( F_{\lambda }\right) \subset F$.

\item If $\rho _{F_{\lambda }}(F)\rightarrow \mathcal{O}$ then $F\subset
\mathcal{LI}\left( F_{\lambda }\right) $.

\item If $\left( F_{\lambda }\right) $ is a decreasing net and $%
F=\bigcap_{\lambda \in \Lambda }{F_{\lambda }}$, with $\rho _{F}\left(
F_{\lambda }\right) \rightarrow \mathcal{O}$, then $F_{\lambda }\rightarrow
F $.

\item Suppose that $X$ is complete and $\left( F_{\lambda }\right) $ is a
net of nonempty closed subsets of $X$ such that $\gamma \left( F_{\lambda
}\right) \rightarrow \mathcal{O}$. Then $F=\bigcap_{\lambda \in \Lambda }{%
F_{\lambda }}$ is nonempty, compact, and $\rho _{H}\left( F_{\lambda
},F\right) \rightarrow \mathcal{O}$.
\end{enumerate}
\end{theorem}

\begin{proof}
$\left( 1\right) $ Take $x\in \mathcal{LS}\left( F_{\lambda }\right) $ and $%
\mathcal{U},\mathcal{V}\in \mathcal{O}$ with $\mathcal{V}\leqslant \frac{1}{2%
}\mathcal{U}$. Since $\rho _{F}\left( F_{\lambda }\right) \rightarrow
\mathcal{O}$, there is $\lambda _{0}\in \Lambda $ such that $\mathcal{V}\in
\rho _{F}\left( F_{\lambda }\right) $ for all $\lambda \geq \lambda _{0}$.
Hence $F_{\lambda }\subset \mathrm{St}\left[ F,\mathcal{V}\right] $ for all $%
\lambda \geq \lambda _{0}$. Since $x\in \mathcal{LS}\left( F_{\lambda
}\right) $, there exists $\lambda ^{\prime }\geq \lambda _{0}$ such that $%
\mathrm{St}\left[ x,\mathcal{V}\right] \cap F_{\lambda ^{\prime }}\neq
\emptyset $. Let $y\in \mathrm{St}\left[ x,\mathcal{V}\right] \cap
F_{\lambda ^{\prime }}$. As $y\in F_{\lambda ^{\prime }}$, it follows that $%
y\in \mathrm{St}\left[ F,\mathcal{V}\right] $, and therefore $\mathcal{V}\in
\rho \left( y,k\right) $ for some $k\in F$. Then we have
\begin{equation*}
\rho (x,k)\prec 1(\rho (x,y)\cap \rho (y,k))\prec 1\{\mathcal{V}\}\prec \{%
\mathcal{U}\}
\end{equation*}%
for all $\mathcal{U}\in \mathcal{O}$. This implies $\rho \left( x,F\right) =%
\mathcal{O}$, and then $x\in F$.

$\left( 2\right) $ Let $x\in F$ and $\mathcal{U}\in \mathcal{O}$. Since $%
\rho _{F_{\lambda }}\left( F\right) \rightarrow \mathcal{O}$, we can take $%
\lambda _{0}\in \Lambda $ such that $\mathcal{U}\in \rho _{F_{\lambda
}}\left( F\right) $ for all $\lambda \geq \lambda _{0}$. Thus $x\in \mathrm{%
St}\left[ F_{\lambda },\mathcal{U}\right] $, which means $F_{\lambda }\cap
\mathrm{St}\left[ x,\mathcal{U}\right] \neq \emptyset $ for all $\lambda
\geq \lambda _{0}$. Therefore $x\in \mathcal{LI}(F_{\lambda })$.

$\left( 3\right) $ It is obvious since $\rho _{F_{\lambda }}\left( F\right) =%
\mathcal{O}$ for every $\lambda $.

$\left( 4\right) $ By Theorem \ref{TK}, $F=\bigcap_{\lambda \in \Lambda }{%
F_{\lambda }}$ is nonempty and compact. Suppose by contradiction that $%
\left( F_{\lambda }\right) $ does not converges to $F$ in the uniform
topology. Then there is some $\mathcal{U}\in \mathcal{O}$ such that for
every $\lambda \in \Lambda $ there exists $\lambda ^{\prime }\geq \lambda $
with $\mathcal{U}\notin \rho _{H}\left( F_{\lambda ^{\prime }},F\right) $.
Since $F\subset F_{\lambda ^{\prime }}$, this means that $F_{\lambda
^{\prime }}\nsubseteq \mathrm{St}\left[ F,\mathcal{U}\right] $. As $%
F_{\lambda ^{\prime }}\subset F_{\lambda }$ it follows that $F_{\lambda
}\nsubseteq \mathrm{St}\left[ F,\mathcal{U}\right] $ for all $\lambda \in
\Lambda $. We now define $C_{\lambda }=F_{\lambda }\cap \left( X\setminus
\mathrm{St}\left[ F,\mathcal{U}\right] \right) $ for each $\lambda \in
\Lambda $. Then $\left( C_{\lambda }\right) $ is a decreasing net of
nonempty closed sets such that $\gamma \left( C_{\lambda }\right)
\rightarrow \mathcal{O}$, as $C_{\lambda }\subset F_{\lambda }$ and $\gamma
\left( F_{\lambda }\right) \rightarrow \mathcal{O}$. By Theorem \ref{TK}, $%
\bigcap_{\lambda \in \Lambda }{C_{\lambda }}\neq \emptyset $. On the other
hand
\begin{equation*}
\bigcap_{\lambda \in \Lambda }{C_{\lambda }}=\bigcap_{\lambda \in \Lambda }{%
F_{\lambda }\cap }\left( X\setminus \mathrm{St}\left[ F,\mathcal{U}\right]
\right) =\bigcap_{\lambda \in \Lambda }{F_{\lambda }}\cap \left( X\setminus
\mathrm{St}\left[ \bigcap_{\lambda \in \Lambda }{F_{\lambda }},\mathcal{U}%
\right] \right) =\emptyset
\end{equation*}%
and then we have a contradiction.
\end{proof}

This theorem allows to relate Hausdorff convergence and Kuratowski
convergence.

\begin{corollary}
Every convergent net in $\mathcal{H}\left( X\right) $ is Kuratowski
convergent.
\end{corollary}

\begin{proof}
If $F_{\lambda }\rightarrow F$ then $\rho _{F}\left( F_{\lambda }\right)
\rightarrow \mathcal{O}$ and $\rho _{F_{\lambda }}\left( F\right)
\rightarrow \mathcal{O}$. By Theorem \ref{PK}, we have $\mathcal{LS}\left(
F_{\lambda }\right) \subset F\subset \mathcal{LI}\left( F_{\lambda }\right) $%
.
\end{proof}

We now show that Hausdorff convergence and Kuratowski convergence are
equivalent for compact space.

\begin{theorem}
\label{KP} Assume that $X$ is a compact admissible space. Let $\left(
K_{\lambda }\right) $ be a net in $\mathcal{H}\left( X\right) $. If $%
K_{\lambda }\overset{k}{\rightarrow }K$ then $K_{\lambda }\rightarrow K$.
\end{theorem}

\begin{proof}
Let $\mathcal{U}\in \mathcal{O}$ and take $x\in X\setminus \mathrm{St}\left[
K,\mathcal{U}\right] $. We have $x\notin \mathcal{LS}\left( K_{\lambda
}\right) =K$. Thus there is an open neighborhood $U_{x}$ of $x$ and $\lambda
_{x}\in \Lambda $ such that $U_{x}\cap K_{\lambda }=\emptyset $ for all $%
\lambda \geq \lambda _{x}$. Since $X\setminus \mathrm{St}\left[ K,\mathcal{U}%
\right] $ is compact, we can take a finite subcovering $X\setminus \mathrm{St%
}\left[ K,\mathcal{U}\right] \subset \bigcup_{i=1}^{n}{U_{x_{i}}}$. Now,
choose $\lambda _{0}\in \Lambda $ such that $\lambda _{0}\geq \lambda
_{x_{i}}$ for all $i=1,...,n$. We have $\bigcup_{i=1}^{n}{U_{x_{i}}}\cap
K_{\lambda _{0}}=\emptyset $. This means that $X\setminus \mathrm{St}\left[
K,\mathcal{U}\right] \cap K_{\lambda }=\emptyset $ whenever $\lambda \geq
\lambda _{0}$. It follows that $K_{\lambda }\subset \mathrm{St}\left[ K,%
\mathcal{U}\right] $ for all $\lambda \geq \lambda _{0}$, and therefore $%
\mathcal{U}\in \rho _{K}\left( K_{\lambda }\right) $ whenever $\lambda \geq
\lambda _{0}$. By the compactness of $K$, we can take a finite open covering
$K\subset \bigcup_{j=i}^{m}U_{j}$ such that $\mathcal{U}\in \mathrm{D}\left(
U_{j}\right) $ and $U_{j}\cap K\neq \emptyset $ for all $j=1,...,m$. On the
other hand, $\mathcal{LI}\left( K_{\lambda }\right) =K$, hence there exists $%
\lambda _{j}\in \Lambda $ for each $j\in \left\{ 1,...,k\right\} $ such that
$U_{j}\cap K_{\lambda }\neq \emptyset $ whenever $\lambda \geq \lambda _{j}$%
. Fix some $\lambda ^{\prime }\geq \lambda _{j}$ for all $j\in \left\{
1,...,k\right\} $. We claim that $K\subset \mathrm{St}\left[ K_{\lambda },%
\mathcal{U}\right] $ for all $\lambda \geq \lambda ^{\prime }$. In fact, for
a given $x\in K$ there is $j\in \left\{ 1,...,k\right\} $ such that $x\in
U_{j}$. Since $\lambda \geq \lambda _{j}$ for all $j$, we have $U_{j}\cap
K_{\lambda }\neq \emptyset $. Hence there exists $y\in U_{j}\cap K_{\lambda
} $, and then $\rho \left( x,y\right) \prec \mathrm{D}\left( U_{j}\right)
\prec \left\{ \mathcal{U}\right\} $. Therefore $K\subset \mathrm{St}\left[
K_{\lambda },\mathcal{U}\right] $ for all $\lambda \geq \lambda ^{\prime }$.
Finally, choose $\lambda _{0}^{\prime }\geq \lambda _{0},\lambda ^{\prime }$%
. This implies that $K\subset \mathrm{St}\left[ K_{\lambda },\mathcal{U}%
\right] $ and $K_{\lambda }\subset \mathrm{St}\left[ K,\mathcal{U}\right] $
for every $\lambda \geq \lambda _{0}^{\prime }$. In other words, $\mathcal{U}%
\in \rho _{K}\left( K_{\lambda }\right) \cap \rho _{K_{\lambda }}\left(
K\right) $ for all $\lambda \geq \lambda _{0}^{\prime }$, and therefore $%
\rho _{H}\left( K_{\lambda },K\right) \rightarrow \mathcal{O}$.
\end{proof}

We now define some notions of continuity of set-valued functions. For a
given set-valued function $F:X\rightarrow \mathcal{H}\left( X\right) $ and a
set $A\subset X$, we define the sets
\begin{eqnarray*}
F^{+}\left( A\right) &=&\left\{ x\in X:F(x)\subset A\right\} , \\
F^{-}\left( A\right) &=&\left\{ x\in X:F(x)\cap A\neq \emptyset \right\} .
\end{eqnarray*}

\begin{definition}
A set-valued function $F:X\rightarrow \mathcal{H}\left( X\right) $ is called

\begin{enumerate}
\item \textbf{lower semicontinuous} (LSC) at $x\in X$ if for every open set $%
V\subset X$ such that $F\left( x\right) \cap V\neq \emptyset $, there is a
neighborhood $U$ of $x$ such that $F\left( y\right) \cap V\neq \emptyset $
for every $y\in U$, that is, $U\subset F^{-}\left( V\right) $.

\item \textbf{upper semicontinuous} (USC) at $x\in X$ if for every open set $%
V\subset X$ such that $F\left( x\right) \subset V$, there is a neighborhood $%
U$ of $x$ such that $F\left( y\right) \subset V$ for every $y\in U$, that
is, $U\subset F^{+}\left( V\right) $.

\item \textbf{Hausdorff continuous} at $x\in X$ if it is lower and upper
semicontinuous at $x$.
\end{enumerate}

We say that $F$ is upper semicontinuous (lower semicontinuous, Hausdorff
continuous) if it is upper semicontinuous (lower semicontinuous, Hausdorff
continuous) at every point of $X$.
\end{definition}

\begin{remark}
Let $F,F_{1},F_{2}:X\rightarrow \mathcal{H}\left( X\right) $ be set-valued
functions such that $F\left( x\right) =F_{1}\left( x\right) \cup F_{2}\left(
x\right) $ for some $x\in X$. If both $F_{1}$ and $F_{2}$ are USC at $x$
then $F$ is USC at $x$. If both $F_{1}$ and $F_{2}$ are closed functions
which are LSC at $x$ then $F$ is LSC at $x$.
\end{remark}

\begin{proposition}
\label{P6} Let $F:X\rightarrow \mathcal{H}\left( X\right) $ be a set-valued
function with $F\left( x\right) $ compact for every $x\in X$. The following
statements holds:

\begin{enumerate}
\item $F$ is USC at $x$ if and only if $\rho _{F(x)}\left( F(x_{\lambda
})\right) \rightarrow \mathcal{O}$ for any net $x_{\lambda }\rightarrow x$.

\item $F$ is LSC at $x$ if and only if $\rho _{F(x_{\lambda })}\left(
F(x)\right) \rightarrow \mathcal{O}$ for any net $x_{\lambda }\rightarrow x$.

\item $F$ is Hausdorff continuous at $x$ if and only if $\rho _{H}\left(
F(x_{\lambda }),F(x)\right) \rightarrow \mathcal{O}$ for any net $x_{\lambda
}\rightarrow x$.
\end{enumerate}
\end{proposition}

\begin{proof}
$\left( 1\right) $ Suppose that $F$ is USC at $x$. Let $x_{\lambda
}\rightarrow x$ and $\mathcal{U}\in \mathcal{O}$. There is $\mathcal{V}\in
\mathcal{O}$ such that $F\left( y\right) \subset \mathrm{St}\left[ F\left(
x\right) ,\mathcal{U}\right] $ for every $y\in \mathrm{St}\left[ x,\mathcal{V%
}\right] $. Take $\lambda _{0}$ such that $x_{\lambda }\in \mathrm{St}\left[
x,\mathcal{V}\right] $ whenever $\lambda \geq \lambda _{0}$. Then $F\left(
x_{\lambda }\right) \subset \mathrm{St}\left[ F\left( x\right) ,\mathcal{U}%
\right] $, and hence $\mathcal{U}\in \rho _{F(x)}\left( F(x_{\lambda
})\right) $ for all $\lambda \geq \lambda _{0}$. Thus $\rho _{F(x)}\left(
F(x_{\lambda })\right) \rightarrow \mathcal{O}$. On the other hand, suppose
that $\rho _{F(x)}\left( F(x_{\lambda })\right) \rightarrow \mathcal{O}$ for
any net $x_{\lambda }\rightarrow x$. Let $V\subset X$ be an open set such
that $F\left( x\right) \subset V$. By the compactness of $F\left( x\right) $%
, there is $\mathcal{V}\in \mathcal{O}$ such that $\mathrm{St}\left[ F\left(
x\right) ,\mathcal{V}\right] \subset V$. Suppose by contradiction that for
every $\mathcal{U}\in \mathcal{O}$ there is $x_{\mathcal{U}}\in \mathrm{St}%
\left[ x,\mathcal{U}\right] $ such that $F\left( x_{\mathcal{U}}\right)
\nsubseteq V$. By hypothesis, $\rho _{F(x)}\left( F(x_{\mathcal{U}})\right)
\rightarrow \mathcal{O}$ since $x_{\mathcal{U}}\rightarrow x$. Then there is
$\mathcal{U}_{0}\in \mathcal{O}$ such that $\mathcal{V}\in \rho
_{F(x)}\left( F(x_{\mathcal{U}})\right) $ whenever $\mathcal{U}\leqslant
\mathcal{U}_{0}$. Hence $F(x_{\mathcal{U}})\subset \mathrm{St}\left[ F\left(
x\right) ,\mathcal{V}\right] \subset V$ for $\mathcal{U}\leqslant \mathcal{U}%
_{0}$, which is a contradiction. Thus $F$ is USC at $x$.

$\left( 2\right) $ Suppose that $F$ is LSC at $x$. Let $x_{\lambda
}\rightarrow x$ and $\mathcal{U}\in \mathcal{O}$. Take $\mathcal{V}\in
\mathcal{O}$ with $\mathcal{V}\leqslant \frac{1}{2}\mathcal{U}$. By the
compactness of $F\left( x\right) $, we can get a finite sequence $%
y_{1},...,y_{n}\in F\left( x\right) $ such that $F\left( x\right) \subset
\bigcup_{i=1}^{n}\mathrm{St}\left[ y_{i},\mathcal{V}\right] $. Since $F$ is
LSC at $x$, there is a neighborhood $U$ of $x$ such that $F\left( y\right)
\cap \mathrm{St}\left[ y_{i},\mathcal{V}\right] \neq \emptyset $ for all $%
y\in U$ and $i=1,...,n$. Take $\lambda _{0}$ such that $x_{\lambda }\in U$
whenever $\lambda \geq \lambda _{0}$. For $y\in F\left( x\right) $ and $%
\lambda \geq \lambda _{0}$, we have $y\in \mathrm{St}\left[ y_{i},\mathcal{V}%
\right] $, for some $i$, and $F\left( x_{\lambda }\right) \cap \mathrm{St}%
\left[ y_{i},\mathcal{V}\right] \neq \emptyset $. As $\mathcal{V}\leqslant
\frac{1}{2}\mathcal{U}$, it follows that $\mathcal{U}\in \rho \left(
y,F(x_{\lambda })\right) $. Hence $\mathcal{U}\in \bigcap_{y\in F\left(
x\right) }\rho \left( y,F(x_{\lambda })\right) $ whenever $\lambda \geq
\lambda _{0}$, and therefore $\rho _{F(x_{\lambda })}\left( F(x)\right)
\rightarrow \mathcal{O}$. Conversely, suppose that $\rho _{F(x_{\lambda
})}\left( F(x)\right) \rightarrow \mathcal{O}$ for any net $x_{\lambda
}\rightarrow x$ and $F$ is not LSC at $x$. Then there is an open set $%
V\subset X$ such that $F\left( x\right) \cap V\neq \emptyset $ and $F\left(
x_{\lambda }\right) \cap V=\emptyset $ for some net $x_{\lambda }\rightarrow
x$. Choose $y\in F\left( x\right) \cap V$ and take $\mathcal{U}\in \mathcal{O%
}$ such that $\mathrm{St}\left[ y,\mathcal{U}\right] \subset V$. By
hypothesis, $\rho _{F(x_{\lambda })}\left( F(x)\right) \rightarrow \mathcal{O%
}$. Hence there is $\lambda _{0}$ such that $\lambda \geq \lambda _{0}$
implies $\mathcal{U}\in \rho _{F(x_{\lambda })}\left( F(x)\right) $. It
follows that $\mathcal{U}\in \rho \left( y,F(x_{\lambda })\right) $ whenever
$\lambda \geq \lambda _{0}$, and then $\emptyset \neq F\left( x_{\lambda
}\right) \cap \mathrm{St}\left[ y,\mathcal{U}\right] \subset F\left(
x_{\lambda }\right) \cap V$, a contradiction.

$\left( 3\right) $ It follows by applying item $\left( 1\right) $ together
with $\left( 2\right) $.
\end{proof}

We after need the following technical result.

\begin{proposition}
\label{LH} Let $F:X\rightarrow \mathcal{H}(X)$ be a function with $F\left(
x\right) $ compact. If $F$ is USC at $x$ then $F(x)\cup \bigcup_{\lambda \in
\Lambda }{F(x_{\lambda })}$ is compact for any convergent net $x_{\lambda
}\rightarrow x$.
\end{proposition}

\begin{proof}
Let $\left( y_{\mu }\right) $ be a net in $\bigcup_{\lambda \in \Lambda }{F}%
\left( {x_{\lambda }}\right) $. Assume that $y_{\mu }\in F\left( x_{\lambda
_{\mu }}\right) $. Since $\rho _{F(x)}\left( F(x_{\lambda _{\mu }})\right)
\rightarrow \mathcal{O}$, we have $\rho \left( y_{\mu },F(x)\right)
\rightarrow \mathcal{O}$. As $F\left( x\right) $ is compact, we can find a
subnet $\left( y_{\mu _{\sigma }}\right) $ and $y\in F\left( x\right) $ so
that $y_{\mu _{\sigma }}\rightarrow y$, by Proposition \ref{PKS}. Therefore $%
\bigcup_{\lambda \in \Lambda }{F}\left( {x_{\lambda }}\right) $ is compact.
\end{proof}

\section{\label{s4}Topological dynamics}

We now apply the previous results on hyperconvergence to topological
dynamics. We study semicontinuity and Hausdorff continuity of set-valued
functions defined by orbit closure, limit set, prolongation, and
prolongational limit set. Throughout, there is a fixed admissible space $X$
endowed with an admissible family of open coverings $\mathcal{O}$.

Let $S$ be a topological semigroup. An \emph{action} of $S$ on $X$ is a
continuous mapping

\begin{equation*}
\mu :%
\begin{array}[t]{ccc}
S\times X & \rightarrow & X \\
(s,x) & \mapsto & \mu (s,x)=sx%
\end{array}%
\end{equation*}%
satisfying $s\left( tx\right) =\left( st\right) x$ for all $x\in X$ and $%
s,t\in S$. We denote by $\mu _{s}:X\rightarrow X$ the map $\mu _{s}\left(
\cdot \right) =\mu \left( s,\cdot \right) $. A subset $Y\subset X$ is said
to be forward invariant if $SY\subset Y$. A subset $M\subset X$ is called
minimal if it is nonempty, closed, forward invariant, and has no proper
subset satisfying these properties. In other words, $M$ is minimal if and
only if $M=\mathrm{cls}\left( Sx\right) $ for every $x\in M$.

For limit behavior of $\left( S,X\right) $, we fix a filter basis $\mathcal{F%
}$ on the subsets of $S$ ($\emptyset \notin \mathcal{F}$ and given $A,B\in
\mathcal{F}$ there is $C\in \mathcal{F}$ with $C\subset A\cap B$). We often
consider $\mathcal{F}$ directed by set inclusion. We might assume that $%
\mathcal{F}$ is a co-compact filter basis, that is, for each $A\in \mathcal{F%
}$, the complement $S\setminus A$ is compact in $S$.

The following notion of stability was stated in \cite{BBRS}.

\begin{definition}
A subset $Y\subset X$ is called \textbf{stable} if for every neighborhood $U$
of $Y$ and $y\in Y$, there is $\mathcal{U}\in \mathcal{O}$ such that $S%
\mathrm{St}\left[ y,\mathcal{U}\right] \subset U$; the set $Y$ is said to be
\textbf{uniformly stable} if for every neighborhood $U$ of $Y$ there $%
\mathcal{U}\in \mathcal{O}$ such that $S\mathrm{St}\left[ Y,\mathcal{U}%
\right] \subset U$.
\end{definition}

\begin{definition}
Let $\mathcal{F}$ be a filter basis on the subsets of $S$. The set $Y$ is
said to be $\mathcal{F}$\textbf{-eventually stable} if for every
neighborhood $U$ of $Y$ there is a neighborhood $V$ of $Y$ such that for
each $x\in V$ one has $Ax\subset U$ for some $A\in \mathcal{F}$.
\end{definition}

\begin{remark}
Every uniformly stable set is stable and $\mathcal{F}$-eventually stable.
Any compact stable set is uniformly stable. If $Y$ is compact then $Y$ is
stable if and only if for every neighborhood $U$ of $Y$ there is a
neighborhood $V$ of $Y$ such that $SV\subset U$.
\end{remark}

The following notion of divergent net was introduced in \cite{Ra}.

\begin{definition}
\label{Note} A net $\left( t_{\lambda }\right) _{\lambda \in \Lambda }$ in $%
S $ \textbf{diverges} on the direction of $\mathcal{F}$ ($\mathcal{F}$%
\textbf{-diverges}) if for each $A\in \mathcal{F}$ there is $\lambda _{0}\in
\Lambda $ such that $t_{\lambda }\in A$ whenever $\lambda \geq \lambda _{0}$%
. The notation $t_{\lambda }\rightarrow _{\mathcal{F}}\infty $ means that $%
\left( t_{\lambda }\right) $ $\mathcal{F}$-diverges.
\end{definition}

The following concept of limit set for semigroup action was introduced in
\cite{BragaSouza}.

\begin{definition}
The $\omega $\textbf{-limit set} of $Y\subset X$ on the direction of $%
\mathcal{F}$ is defined as
\begin{eqnarray*}
\omega \left( Y,\mathcal{F}\right) &=&\bigcap_{A\in \mathcal{F}}\mathrm{cls}%
\left( AY\right) \\
&=&\left\{
\begin{array}{c}
x\in X:\text{ there are nets }\left( t_{\lambda }\right) _{\lambda \in
\Lambda }\text{ in }S\text{ and }\left( x_{\lambda }\right) _{\lambda \in
\Lambda }\text{ in }Y \\
\text{such that }t_{\lambda }\rightarrow _{\mathcal{F}}\infty \text{ and }%
t_{\lambda }x_{\lambda }\rightarrow x%
\end{array}%
\right\} .
\end{eqnarray*}
\end{definition}

This notion of $\omega $-limit set extends the Conley definition of $\omega $%
-limit set for the Morse theory in dynamical systems (\cite{c}). Note that $%
\bigcup_{y\in Y}\omega \left( y,\mathcal{F}\right) \subset \omega \left( Y,%
\mathcal{F}\right) $, but the equality does not hold in general. Actually,
the Conley definition of $\omega $-limit set of a subset approaches to the
notion of prolongational limit set defined afterwards (Remarks \ref{R2} and %
\ref{R3}). We will show that how the continuity of the orbital functions
depends on the equality $\bigcup_{y\in Y}\omega \left( y,\mathcal{F}\right)
=\omega \left( Y,\mathcal{F}\right) $, and vice-versa.

We might assume the following additional hypothesis on the family $\mathcal{F%
}$.

\begin{definition}
\label{hipH} The family $\mathcal{F}$ is said to satisfy:

\begin{enumerate}
\item Hypothesis $\mathrm{H}_{1}$ if for all $s\in S$ and $A\in \mathcal{F}$
there exists $B\in \mathcal{F}$ such that $sB\subset A$.

\item Hypothesis $\mathrm{H}_{2}$ if for all $s\in S$ and $A\in \mathcal{F}$
there exists $B\in \mathcal{F}$ such that $Bs\subset A$.

\item Hypothesis $\mathrm{H}_{3}$ if for all $s\in S$ and $A\in \mathcal{F}$
there exists $B\in \mathcal{F}$ such that $B\subset As$.
\end{enumerate}
\end{definition}

Hypothesis $\mathrm{H}_{1}$ yields the limit set $\omega \left( Y,\mathcal{F}%
\right) $ is forward invariant. Hypotheses $\mathrm{H}_{2}$ and $\mathrm{H}%
_{3}$ implies respectively $\omega \left( sx,\mathcal{F}\right) \subset
\omega \left( x,\mathcal{F}\right) $ and $\omega \left( x,\mathcal{F}\right)
\subset \omega \left( sx,\mathcal{F}\right) $ for every $s\in S$ and $x\in X$%
.

The following notion of attraction was introduced in \cite{BBRSCan}.

\begin{definition}
The $\mathcal{F}$\textbf{-domain of attraction} of a set $Y\subset X$ is
defined by
\begin{equation*}
\mathfrak{A}\left( Y,\mathcal{F}\right) =\left\{ x\in X:\text{for each }%
\mathcal{U}\in \mathcal{O}\text{ there is }A\in \mathcal{F}\text{ such that }%
Ax\subset \mathrm{St}\left[ Y,\mathcal{U}\right] \right\} .
\end{equation*}%
The set $Y$ is called $\mathcal{F}$\textbf{-attractor} if there is $\mathcal{%
U}\in \mathcal{O}$ such that $\mathrm{St}\left[ Y,\mathcal{U}\right] \subset
\mathfrak{A}\left( Y,\mathcal{F}\right) $; it is called \textbf{global }$%
\mathcal{F}$\textbf{-attractor} if $\mathfrak{A}\left( Y,\mathcal{F}\right)
=X$.
\end{definition}

\begin{remark}
If $K\subset X$ is compact then
\begin{equation*}
\mathfrak{A}\left( K,\mathcal{F}\right) \subset \left\{ x\in X:\omega \left(
x,\mathcal{F}\right) \neq \emptyset \text{ and }\omega \left( x,\mathcal{F}%
\right) \subset K\right\} .
\end{equation*}%
The equality holds if $X$ is locally compact and $Ax$ is connected for all $%
A\in \mathcal{F}$ and $x\in X$ (\cite[Theorem 3.6]{BBRSCan}).
\end{remark}

\begin{definition}
The semigroup action $\left( S,X\right) $ is $\mathcal{F}$\textbf{-limit
compact} if for every bounded set $Y\subset X$ and any $\mathcal{U}\in
\mathcal{O}$ there is $A\in \mathcal{F}$ such that $\mathcal{U}\in \gamma
\left( AY\right) $.
\end{definition}

The following result is an application of Theorem \ref{PK}.

\begin{proposition}
\label{P4}Assume that $X$ is a complete admissible space. If the semigroup
action $\left( S,X\right) $ is $\mathcal{F}$-limit compact then $\omega
\left( Y,\mathcal{F}\right) $ is nonempty, compact, and the net $\left(
\mathrm{cls}(AY)\right) _{A\in \mathcal{F}}$ converges to $\omega \left( Y,%
\mathcal{F}\right) $, for all bounded set $Y\subset X$.
\end{proposition}

\begin{proof}
Firstly, we prove that $\gamma \left( \mathrm{cls}(AY)\right) \rightarrow
\mathcal{O}$. For a given $\mathcal{U}\in \mathcal{O}$, take $\mathcal{V}\in
\mathcal{O}$ such that $\mathcal{V}\leqslant \frac{1}{2}\mathcal{U}$. Since $%
Y$ is a bounded subset of $X$, there exists $A_{1}\in \mathcal{F}$ such that
$\mathcal{V}\in \gamma \left( A_{1}Y\right) $. By Proposition \ref{P19}, we
have $\gamma \left( \mathrm{cls}(A_{1}Y)\right) \prec 1\left\{ \mathcal{V}%
\right\} \prec \left\{ \mathcal{U}\right\} $. For every $A\subset A_{1}$, we
have $\gamma \left( \mathrm{cls}(AY)\right) \prec \gamma \left( \mathrm{cls}%
(A_{1}Y)\right) $. Hence $\mathcal{U}\in \gamma \left( \mathrm{cls}%
(AY)\right) $ whenever $A\subset A_{1}$, and therefore $\gamma \left(
\mathrm{cls}(AY)\right) \rightarrow \mathcal{O}$. By Theorem \ref{PK}, it
follows that $\bigcap_{A\in \mathcal{F}}\mathrm{cls}\left( AY\right) =\omega
\left( Y,\mathcal{F}\right) $ is nonempty, compact, and $\rho _{H}\left(
\mathrm{cls}(AY),\omega \left( Y,\mathcal{F}\right) \right) \rightarrow
\mathcal{O}$.
\end{proof}

In line of this statement, we define the set-valued functions $K_{A},L_{%
\mathcal{F}}:X\rightarrow \mathcal{H}\left( X\right) $ by
\begin{equation*}
K_{A}\left( x\right) =\mathrm{cls}\left( Ax\right) ,\quad A\subset S,\qquad
L_{\mathcal{F}}\left( x\right) =\bigcap_{A\in \mathcal{F}}K_{A}\left(
x\right) =\omega \left( x,\mathcal{F}\right) .
\end{equation*}

If $X$ is complete and the semigroup action $\left( S,X\right) $ is $%
\mathcal{F}$-limit compact, Proposition \ref{P4} says that the net $\left(
K_{A}\right) _{A\in \mathcal{F}}$ pointwise converges to $L_{\mathcal{F}}$.

We also define the set-valued functions $D_{A},J_{\mathcal{F}}:X\rightarrow
\mathcal{H}\left( X\right) $ by%
\begin{equation*}
D_{A}\left( x\right) =\bigcap_{\mathcal{U}\in \mathcal{O}}{\mathrm{cls}}%
\left( {A\mathrm{St}}\left[ {x,\mathcal{U}}\right] \right) ,\quad A\subset
S,\qquad J_{\mathcal{F}}\left( x\right) =\bigcap_{A\in \mathcal{F}%
}D_{A}\left( x\right) .
\end{equation*}

In the language of semigroup actions, $D_{A}\left( x\right) $ is the $A$%
-prolongation of $x$ and $J_{\mathcal{F}}\left( x\right) $ is the $\mathcal{F%
}$-prolongational limit set of $x$. They are described as%
\begin{eqnarray*}
D_{A}\left( x\right) &=&\left\{
\begin{array}{c}
y\in X:\text{there are nets }\left( t_{\lambda }\right) _{\lambda \in
\Lambda }\text{ in }A\text{ and }\left( x_{\lambda }\right) _{\lambda \in
\Lambda }\text{ in }X\text{ } \\
\text{ such that }x_{\lambda }\rightarrow x\text{ and }t_{\lambda
}x_{\lambda }\rightarrow y%
\end{array}%
\right\} , \\
J_{\mathcal{F}}\left( x\right) &=&\left\{
\begin{array}{c}
y\in X:\text{there are nets }\left( t_{\lambda }\right) _{\lambda \in
\Lambda }\text{ in }S\text{ and }\left( x_{\lambda }\right) _{\lambda \in
\Lambda }\text{ in }X\text{ such that } \\
t_{\lambda }\rightarrow _{\mathcal{F}}\infty \text{, }x_{\lambda
}\rightarrow x\text{, and }t_{\lambda }x_{\lambda }\rightarrow y%
\end{array}%
\right\} .
\end{eqnarray*}%
See \cite{SouzaTozatti} for details.

\begin{remark}
\label{R2}It is not difficult to check that $J_{\mathcal{F}}\left( x\right)
=\bigcap_{\mathcal{U}\in \mathcal{O}}{\omega }\left( {\mathrm{St}[x,\mathcal{%
U}],\mathcal{F}}\right) $ for every $x\in X$.
\end{remark}

\begin{remark}
\label{R3}If $K\subset X$ is compact and $U\subset X$ is open then $\omega
\left( K,\mathcal{F}\right) \subset \bigcup_{x\in K}J_{\mathcal{F}}\left(
x\right) $ and $\bigcup_{x\in U}J_{\mathcal{F}}\left( x\right) \subset
\omega \left( U,\mathcal{F}\right) $ (\cite[Proposition 2.14]{BBRSmn}).
\end{remark}

\begin{remark}
If $\mathcal{F}$ is a co-compact filter basis then $D_{A}\left( x\right)
=K_{A}\left( x\right) \cup J_{\mathcal{F}}\left( x\right) $ for every $x\in
X $. This is a little variation of Theorem 2.2 in \cite[Theorem 2.2]%
{SouzaTozatti}.
\end{remark}

\begin{remark}
If $\mathcal{F}$ satisfies the translation hypotheses $\mathrm{H}_{1},%
\mathrm{H}_{2}$, $Ax$ is connected and $\left( S\setminus A\right) x$ is
closed, for all $x\in X$ and $A\in \mathcal{F}$, then $L_{\mathcal{F}}\left(
x\right) $ is nonempty and compact whenever $J_{\mathcal{F}}\left( x\right) $
is nonempty and compact (\cite[Proposition 3.1]{BBRSmn}).
\end{remark}

For the following, we assume $\mathcal{O}$ directed by refinements.

\begin{proposition}
\label{P5}Assume that $X$ is a complete admissible space and the semigroup
action $\left( S,X\right) $ is $\mathcal{F}$-limit compact. Then $J_{%
\mathcal{F}}\left( x\right) $ is nonempty and compact and the net $\left(
\mathrm{cls}(A\mathrm{St}\left[ x,\mathcal{U}\right] )\right) _{(A,\mathcal{U%
})\in \mathcal{F}\times \mathcal{O}}$ converges to $J_{\mathcal{F}}\left(
x\right) $, for all $x\in X$.
\end{proposition}

\begin{proof}
Consider the product direction on $\mathcal{F}\times \mathcal{O}$: $\left( A,%
\mathcal{U}\right) \leq \left( B,\mathcal{V}\right) $ if $B\subset A$ and $%
\mathcal{V}\leq \mathcal{U}$. We claim that $\gamma \left( A\mathrm{St}[x,%
\mathcal{U}]\right) \rightarrow \mathcal{O}$. Fix some $\mathcal{U}_{0}\in
\mathcal{O}$. For a given $\mathcal{U}\in \mathcal{O}$, take $\mathcal{V}\in
\mathcal{O}$ such that $\mathcal{V}\leq \frac{1}{2}\mathcal{U}$. Since $%
\mathrm{St}\left[ x,\mathcal{U}_{0}\right] $ is bounded, there exits $%
A_{0}\in \mathcal{F}$ such that $\mathcal{V}\in \gamma \left( A_{0}\mathrm{St%
}[x,\mathcal{U}_{0}]\right) $, and then $\gamma \left( \mathrm{cls}(A_{0}%
\mathrm{St}[x,\mathcal{U}_{0}])\right) \prec 1\left\{ \mathcal{V}\right\}
\prec \left\{ \mathcal{U}\right\} $, by Proposition \ref{P19}. Moreover, we
have $\mathrm{cls}\left( A\mathrm{St}[x,\mathcal{W}]\right) )\subset \mathrm{%
cls}\left( A_{0}\mathrm{St}[x,\mathcal{U}_{0}]\right) $, for all $(A_{0},%
\mathcal{U}_{0})\leq (A,\mathcal{W})$. Thus $\gamma \left( \mathrm{cls}(A%
\mathrm{St}[x,\mathcal{W}])\right) \prec \gamma \left( \mathrm{cls}(A_{0}%
\mathrm{St}[x,\mathcal{U}_{0}])\right) $, which implies $\mathcal{U}\in
\gamma \left( \mathrm{cls}(A\mathrm{St}[x,\mathcal{W}]\right) $ for all $%
\left( A_{0},\mathcal{U}_{0}\right) \leq \left( A,\mathcal{W}\right) $. By
Proposition \ref{PK}, it follows that $J_{\mathcal{F}}\left( x\right)
=\bigcap_{A\in \mathcal{F},\mathcal{U}\in \mathcal{O}}{\mathrm{cls}}\left( {A%
\mathrm{St}[x,\mathcal{U}]}\right) $ is nonempty, compact, and
\begin{equation*}
\rho _{H}(\mathrm{cls}(A\mathrm{St}[x,\mathcal{U}]),D(x,A))\rightarrow
\mathcal{O}.
\end{equation*}
\end{proof}

In the sequence, we study the continuity of the set-valued functions $K_{A}$%
, $D_{A}$, $L_{\mathcal{F}}$, and $J_{\mathcal{F}}$. We assume hereon that $%
K_{A}\left( x\right) $, $D_{A}\left( x\right) $, $L_{\mathcal{F}}\left(
x\right) $, and $J_{\mathcal{F}}\left( x\right) $ are nonempty and compact
for every $x\in X$.

\begin{proposition}
\label{KA}The function $K_{A}$ is LSC.
\end{proposition}

\begin{proof}
Let $x\in X$ and $U\subset X$ be an open neighborhood of $x$. Take $y\in
K_{A}^{-}\left( U\right) $. Then $K_{A}\left( y\right) \cap U\neq \emptyset $%
, and hence exists $ty\in Ay\cap U$ with $t\in A$. By continuity, we can
take $\mathcal{U}\in \mathcal{O}$ such that $t\mathrm{St}\left[ y,\mathcal{U}%
\right] \subset U$. Now, for any $z\in \mathrm{St}\left[ y,\mathcal{U}\right]
$, $tz\in U\cap Az$, hence $K_{A}\left( z\right) \cap U\neq \emptyset $.
This means that $z\in K_{A}^{+}\left( U\right) $, and therefore $\mathrm{St}%
\left[ y,\mathcal{U}\right] \subset K_{A}^{+}\left( U\right) $. It follows
that $K_{A}^{+}\left( U\right) $ is an open set and thus $K_{A}$ is LSC at $%
x $.
\end{proof}

By Proposition \ref{P6}, we have $\rho _{K_{A}(x_{\lambda })}\left(
K_{A}(x)\right) \rightarrow \mathcal{O}$ for any net $x_{\lambda
}\rightarrow x$. The following is an immediate consequence of Proposition %
\ref{KA}.

\begin{corollary}
\label{Co1}The following statements are equivalent.

\begin{enumerate}
\item $K_{A}$ is USC at $x\in X$.

\item $K_{A}$ is continuous at $x$.
\end{enumerate}
\end{corollary}

\begin{theorem}
The map $D_{A}$ is USC at $x$ if and only if the reunion $D_{A}\left(
x\right) \cup \bigcup_{\lambda \in \Lambda }{D_{A}}\left( {x_{\lambda }}%
\right) $ is compact for every net $x_{\lambda }\rightarrow x$.
\end{theorem}

\begin{proof}
Suppose that $D_{A}\left( x\right) \cup \bigcup_{\lambda \in \Lambda }{D_{A}}%
\left( {x_{\lambda }}\right) $ is compact for every net $x_{\lambda
}\rightarrow x$. Suppose by contradiction that $D_{A}$ is not USC at $x$.
Then there exists a net $x_{\lambda }\rightarrow x$ such that $\rho
_{D_{A}(x)}\left( D_{A}(x_{\lambda })\right) $ does not converges to $%
\mathcal{O}$. Hence there exists some $\mathcal{U}\in \mathcal{O}$ such that
for every $\lambda $ there is $\lambda ^{\prime }\geq \lambda $ such that $%
\mathcal{U}\notin \rho _{D_{A}(x)}\left( D_{A}(x_{\lambda ^{\prime
}})\right) $. Then we can take $y_{\lambda }\in D_{A}\left( x_{\lambda
^{\prime }}\right) $ such that $\mathcal{U}\notin \rho \left( y_{\lambda
},D_{A}(x)\right) $, that is, $y_{\lambda }\notin \mathrm{St}\left[
D_{A}\left( x\right) ,\mathcal{U}\right] $. Now, since $D_{A}\left( x\right)
\cup \bigcup_{\lambda \in \Lambda }{D_{A}}\left( {x_{\lambda }}\right) $ is
compact by hypothesis, we may assume that the net $\left( y_{\lambda
}\right) $ converges to some point $y$. As $y_{\lambda }\in X\setminus
\mathrm{St}\left[ D_{A}\left( x\right) ,\mathcal{U}\right] $, it follows
that $y\in X\setminus \mathrm{St}\left[ D_{A}\left( x\right) ,\mathcal{U}%
\right] $. We claim this is a contradiction. Indeed, for arbitraries $%
\mathcal{V},\mathcal{W}\in \mathcal{O}$, we can choose a $\lambda $ such
that $x_{\lambda ^{\prime }}\in \mathrm{St}\left[ x,\mathcal{V}\right] $ and
$y_{\lambda }\in \mathrm{St}\left[ y,\mathcal{W}\right] $. Take $\mathcal{V}%
^{\prime }\in \mathcal{O}$ such that $\mathrm{St}\left[ x_{\lambda ^{\prime
}},\mathcal{V}^{\prime }\right] \subset \mathrm{St}\left[ x,\mathcal{V}%
\right] $. As $y_{\lambda }\in D_{A}\left( x_{\lambda ^{\prime }}\right) $,
it follows that $\mathrm{St}\left[ y,\mathcal{W}\right] \cap A\mathrm{St}%
\left[ x_{\lambda ^{\prime }},\mathcal{V}^{\prime }\right] \neq \emptyset $,
and hence $\mathrm{St}\left[ y,\mathcal{W}\right] \cap A\mathrm{St}\left[ x,%
\mathcal{V}\right] \neq \emptyset $. This means that $y\in D_{A}\left(
x\right) $, which contradicts $y\in X\setminus \mathrm{St}\left[ D_{A}\left(
x\right) ,\mathcal{U}\right] $. Thus $D_{A}$ is USC at $x$. The converse
follows by Proposition \ref{LH}.
\end{proof}

\begin{theorem}
\label{T6}Assume that $\mathcal{F}$ is a filter basis on the connected
subsets of $S$. If $X$ is Hausdorff locally compact then $J_{\mathcal{F}}$
is USC.
\end{theorem}

\begin{proof}
For a given $\mathcal{U}\in \mathcal{O}$, we take an open set $U\subset X$
such that $\mathrm{cls}\left( U\right) $ is compact and $J_{\mathcal{F}%
}\left( x\right) \subset U\subset \mathrm{cls}\left( U\right) \subset
\mathrm{St}\left[ J_{\mathcal{F}}\left( x\right) ,\mathcal{U}\right] $.

\begin{claim}
There exist $\mathcal{W}\in \mathcal{O}$ and $A\in \mathcal{F}$ such that $A%
\mathrm{St}\left[ x,\mathcal{W}\right] \subset U$. In fact, since $\mathrm{fr%
}\left( \mathrm{cls}\left( U\right) \right) $ is compact there exist $%
A_{0}\in \mathcal{F}$ and $\mathcal{U}_{0}\in \mathcal{O}$ such that $%
\mathrm{fr}\left( \mathrm{cls}\left( U\right) \right) \cap A_{0}\mathrm{St}%
\left[ x,\mathcal{U}_{0}\right] =\emptyset $. Suppose to the contrary that
we can obtain $t_{(A,\mathcal{W})}x_{(A,\mathcal{W})}\in \mathrm{fr}\left(
\mathrm{cls}\left( U\right) \right) \cap A\mathrm{St}\left[ x,\mathcal{W}%
\right] $, with $t_{(A,\mathcal{W})}\in A$ and $x_{(A,\mathcal{W})}\in
\mathrm{St}\left[ x,\mathcal{W}\right] $, for every $A\in \mathcal{F}$ and $%
\mathcal{W}\in \mathcal{O}$. By compactness, we may assume that $t_{(A,%
\mathcal{W})}x_{(A,\mathcal{W})}\rightarrow y$ for some $y\in \mathrm{fr}%
\left( \mathrm{cls}\left( U\right) \right) $. As $t_{(A,\mathcal{W}%
)}\rightarrow _{\mathcal{F}}\infty $ and $x_{(A,\mathcal{W})}\rightarrow x$
we have $y\in J_{\mathcal{F}}\left( x\right) $, which is impossible. Now,
for $t\in A_{0}$ and $y\in \mathrm{St}\left[ x,\mathcal{U}_{0}\right] $, we
have two possibility: either$\ ty\in U$, in which $A_{0}y\subset U$ by
connectedness of $A_{0}y$, or $ty\in X\setminus \mathrm{cls}\left( U\right) $%
, in which $A_{0}y\subset X\setminus \mathrm{cls}\left( U\right) $. Then we
have the equalities
\begin{eqnarray*}
S_{0} &=&\left\{ y\in \mathrm{St}\left[ x,\mathcal{U}_{0}\right]
:A_{0}y\subset U\right\} =\mathrm{St}\left[ x,\mathcal{U}_{0}\right] \cap
\mu _{t}^{-1}\left( U\right) , \\
S_{1} &=&\left\{ y\in \mathrm{St}\left[ x,\mathcal{U}_{0}\right]
:A_{0}y\subset X\setminus \mathrm{cls}\left( U\right) \right\} =\mathrm{St}%
\left[ x,\mathcal{U}_{0}\right] \cap \mu _{t}^{-1}\left( X\setminus \mathrm{%
cls}\left( U\right) \right) ,
\end{eqnarray*}%
which imply the sets $S_{0},S_{1}$ are open. Suppose that $x\in S_{1}$ and
take $\mathcal{V}\in \mathcal{O}$ such that $\mathrm{St}\left[ x,\mathcal{V}%
\right] \subset S_{1}$. As $\mathrm{cls}\left( A_{0}S_{1}\right) \subset
X\setminus U$, it follows that $J_{\mathcal{F}}\left( x\right) \subset
\mathrm{cls}\left( A_{0}\mathrm{St}\left[ x,\mathcal{V}\right] \right)
\subset X\setminus U$, which is impossible. Hence $x\notin S_{1}$, and then $%
tx\notin X\setminus \mathrm{cls}\left( U\right) $. This means that $tx\in U$
and therefore $x\in S_{0}$. Then consider $\mathcal{W}_{0}\in \mathcal{O}$
such that $\mathrm{St}\left[ x,\mathcal{W}_{0}\right] \subset S_{0}$. This
implies that $A_{0}\mathrm{St}\left[ x,\mathcal{W}_{0}\right] \subset
A_{0}S_{0}\subset U$. This proves the claim.
\end{claim}

Finally, for a given $x^{\prime }\in \mathrm{St}\left[ x,\mathcal{W}_{0}%
\right] $, let $\mathcal{W}_{0}^{\prime }\in \mathcal{O}$ with $\mathrm{St}%
\left[ x^{\prime },\mathcal{W}_{0}^{\prime }\right] \subset \mathrm{St}\left[
x,\mathcal{W}_{0}\right] $. We have%
\begin{equation*}
J_{\mathcal{F}}\left( x^{\prime }\right) \subset {\mathrm{cls}(A_{0}\mathrm{%
St}\left[ x^{\prime },\mathcal{W}_{0}^{\prime }\right] )}\subset \mathrm{cls}%
(A_{0}\mathrm{St}\left[ x,\mathcal{W}_{0}\right] )\subset \mathrm{cls}%
(U)\subset \mathrm{St}\left[ J_{\mathcal{F}}\left( x\right) ,\mathcal{U}%
\right] .
\end{equation*}%
Consequently, $\mathcal{U}\in \rho _{J_{\mathcal{F}}\left( x\right) }\left(
J_{\mathcal{F}}\left( x^{\prime }\right) \right) $ whenever $\mathcal{W}%
_{0}\in \rho \left( x^{\prime },x\right) $. Therefore $J$ is USC at $x\in X$.
\end{proof}

Under the conditions of Theorem \ref{T6}, we have $\rho _{J_{\mathcal{F}%
}(x)}\left( J_{\mathcal{F}}(x_{\lambda })\right) \rightarrow \mathcal{O}$
for any net $x_{\lambda }\rightarrow x$. By taking in particular $\mathcal{F}%
=\left\{ A\right\} $, we have the following consequence.

\begin{corollary}
If $X$ is Hausdorff locally compact and $A\subset S$ is nonempty and
connected then $D_{A}$ is USC. Thus $D_{A}$ is Hausdorff continuous if and
only if $D_{A}$ is LSC.
\end{corollary}

The following result means that the upper semicontinuity causes orbits with
trivial prolongations.

\begin{proposition}
\label{K} If $K_{A}$ is USC at $x$ then $K_{A}\left( x\right) =D_{A}\left(
x\right) $. The converse holds if $X$ is Hausdorff locally compact and $A$
is connected.
\end{proposition}

\begin{proof}
Suppose that $K_{A}$ is USC at $x$ and let $y\in D_{A}(x)$. Then there exist
nets $\left( t_{\lambda }\right) _{\lambda \in \Lambda }$ in $A$ and $\left(
x_{\lambda }\right) _{\lambda \in \Lambda }$ in $X$ such that $x_{\lambda
}\rightarrow x$ and $t_{\lambda }x_{\lambda }\rightarrow y$. Since $%
t_{\lambda }x_{\lambda }\in K_{A}(x_{\lambda })$ and $\rho _{K_{A}(x)}\left(
K_{A}(x_{\lambda })\right) \rightarrow \mathcal{O}$, it follows that $\rho
_{K_{A}(x)}\left( t_{\lambda }x_{\lambda }\right) \rightarrow \mathcal{O}$.
Since $K_{A}\left( x\right) $ is closed, we obtain $y\in K_{A}\left(
x\right) $, by Proposition \ref{PKS}. The converse follows by Theorem \ref%
{T6}.
\end{proof}

As a consequence of Corollary \ref{Co1} and Proposition \ref{K}, we have the
following.

\begin{corollary}
\label{Co2}Assume that $X$ is Hausdorff locally compact and $A\subset S$ is
nonempty and connected. The following statements are equivalent:

\begin{enumerate}
\item $K_{A}$ is Hausdorff continuous.

\item $K_{A}$ is USC.

\item $K_{A}=D_{A}$.
\end{enumerate}

Any one of these three conditions implies that $D_{A}$ is Hausdorff
continuous.
\end{corollary}

If orbits have trivial prolongations then the prolongational limit sets
reduce to the limit sets, as the following.

\begin{theorem}
\label{T2}If $K_{A}$ is USC for all $A\in \mathcal{F}$ then $J_{\mathcal{F}%
}=L_{\mathcal{F}}$. The converse holds if $X$ is Hausdorff locally compact
and locally connected, and $\mathcal{F}$ is a co-compact filter basis on the
connected subsets of $S$.
\end{theorem}

\begin{proof}
Suppose that $K_{A}$ is USC for all $A\in \mathcal{F}$. By Proposition \ref%
{K} we have
\begin{equation*}
J_{\mathcal{F}}\left( x\right) =\bigcap_{A\in \mathcal{F}}{D}_{A}\left( {x}%
\right) =\bigcap_{A\in \mathcal{F}}{K_{A}}\left( {x}\right) =L_{\mathcal{F}%
}\left( x\right)
\end{equation*}%
for every $x\in X$. As to the converse, suppose that $X$ is locally
connected, $\mathcal{F}$ is a filter basis on the connected subsets of $S$,
and $J_{\mathcal{F}}=L_{\mathcal{F}}$. Let $A\in \mathcal{F}$ and $x\in X$.
For a given open neighborhood $U$ of $K_{A}\left( x\right) $, take a compact
neighborhood $V$ of $K_{A}\left( x\right) $ such that $V\subset U$. We claim
that there is $\mathcal{U}\in \mathcal{O}$ such that $A\mathrm{St}\left[ x,%
\mathcal{U}\right] \subset V$. Indeed, suppose by contradiction that $A%
\mathrm{St}\left[ x,\mathcal{U}\right] \nsubseteq V$ for all $\mathcal{U}\in
\mathcal{O}$. For each $\mathcal{U}\in \mathcal{O}$, take a connected
neighborhood $N_{\mathcal{U}}$ of $x$ with $N_{\mathcal{U}}\subset \mathrm{St%
}\left[ x,\mathcal{U}\right] $. Then $AN_{\mathcal{U}}\nsubseteq V$ since
there is a star of $x$ inside $N_{\mathcal{U}}$. As $Ax\subset V$ and $AN_{%
\mathcal{U}}$ is connected, it follows that $AN_{\mathcal{U}}\cap \mathrm{fr}%
\left( V\right) \neq \emptyset $. Then we can take $t_{\mathcal{U}}x_{%
\mathcal{U}}\in AN_{\mathcal{U}}\cap \mathrm{fr}\left( V\right) $ with $t_{%
\mathcal{U}}\in A$ and $x_{\mathcal{U}}\in N_{\mathcal{U}}$. By the
compactness of $\mathrm{fr}\left( V\right) $, we may assume that $t_{%
\mathcal{U}}x_{\mathcal{U}}\rightarrow y$ for some $y\in \mathrm{fr}\left(
V\right) $. Now, if $t_{\mathcal{U}}\rightarrow _{\mathcal{F}}\infty $ then $%
y\in J_{\mathcal{F}}\left( x\right) =L_{\mathcal{F}}\left( x\right) \subset
K_{A}\left( x\right) $, which contradicts $K_{A}\left( x\right) \subset
\mathrm{int}\left( V\right) $. Hence $\left( t_{\mathcal{U}}\right) $ does
not $\mathcal{F}$-diverges. As $\mathcal{F}$ is co-compact, we may assume
that $t_{\mathcal{U}}\rightarrow t$ with $t\in \mathrm{cls}\left( A\right) $%
. It follows that $t_{\mathcal{U}}x_{\mathcal{U}}\rightarrow tx$, and hence $%
y=tx\in \mathrm{cls}\left( Ax\right) =K_{A}\left( x\right) $, that is again
a contradiction. This proves the claim. Finally, take $\mathcal{U}\in
\mathcal{O}$ such that $A\mathrm{St}\left[ x,\mathcal{U}\right] \subset V$.
For each $z\in \mathrm{St}\left[ x,\mathcal{U}\right] $ we have $K_{A}\left(
z\right) \subset \mathrm{cls}\left( A\mathrm{St}\left[ x,\mathcal{U}\right]
\right) \subset V\subset U$. Therefore $K_{A}$ is USC at $x$.
\end{proof}

We have the following consequence of Theorems \ref{T6} and \ref{T2}.

\begin{corollary}
Assume that $X$ is Hausdorff locally compact and $\mathcal{F}$ is a filter
basis on the connected subsets of $S$. If $K_{A}$ is USC, for every $A\in
\mathcal{F}$, then $L_{\mathcal{F}}$ is USC.
\end{corollary}

We also have the following relation between the continuity of the orbital
functions and the Conley definition of $\omega $-limit set.

\begin{corollary}
If $K_{A}$ is USC for all $A\in \mathcal{F}$ then $\omega \left( K,\mathcal{F%
}\right) =\bigcup_{x\in K}L_{\mathcal{F}}\left( x\right) $ for every compact
set $K\subset X$. The converse holds if $X$ is Hausdorff locally compact and
locally connected, and $\mathcal{F}$ is a co-compact filter basis on the
connected subsets of $S$.
\end{corollary}

\begin{proof}
Suppose that $K_{A}$ is USC for all $A\in \mathcal{F}$ and let $K\subset X$
be a compact set. By Theorem \ref{T2} and Remark \ref{R3}, we have
\begin{equation*}
\bigcup_{x\in K}L_{\mathcal{F}}\left( x\right) \subset \omega \left( K,%
\mathcal{F}\right) \subset \bigcup_{x\in K}J_{\mathcal{F}}\left( x\right)
=\bigcup_{x\in K}L_{\mathcal{F}}\left( x\right) .
\end{equation*}%
As each single set $\left\{ x\right\} $ is compact, the converse follows by
the second part of Theorem \ref{T2}.
\end{proof}

\begin{theorem}
\label{T1} If $K_{S}$ is USC then $K_{S}\left( x\right) $ is stable for
every $x\in X$. The converse holds if $x\in K_{S}\left( x\right) $ for every
$x\in X$.
\end{theorem}

\begin{proof}
Suppose that $K_{S}$ is USC and take $x\in X$. Let $U\subset X$ be an open
neighborhood of $K_{S}\left( x\right) $. If $y\in K_{S}\left( x\right) $
then $K_{S}\left( y\right) \subset K_{S}\left( x\right) \subset U$. Hence
there is an open neighborhood $V$ of $y$ such that $K_{S}\left( z\right)
\subset U$ for all $z\in V$. It follows that $SV\subset U$, and therefore $%
K_{S}\left( x\right) $ is stable. As to the converse, suppose that $%
K_{S}\left( x\right) $ is stable and $x\in K_{S}\left( x\right) $ for every $%
x\in X$. Let $U\subset X$ be an open neighborhood of $K_{S}\left( x\right) $%
. As $K_{S}\left( x\right) $ is compact and stable, there are neighborhoods $%
V,W$ of $K_{S}\left( x\right) $ such that $SV\subset W\subset \mathrm{cls}%
\left( W\right) \subset U$. If $y\in V$, we have $K_{S}\left( y\right)
\subset \mathrm{cls}\left( SV\right) \subset \mathrm{cls}\left( W\right)
\subset U$. Since $V$ is a neighborhood of $x$, this means that $K_{S}$ is
USC at $x$.
\end{proof}

Note that $x\in K_{S}\left( x\right) $ is and only if $x$ is a weak
transitive point. For instance, every point is weak transitive if $S$ has
identity. The following result is a combination of Corollary \ref{Co2} and
Theorem \ref{T1}.

\begin{corollary}
Assume that $X$ is Hausdorff locally compact, $S$ is connected, and $x\in
K_{S}\left( x\right) $ for every $x\in X$. The following statements are
equivalent:

\begin{enumerate}
\item $K_{S}$ is Hausdorff continuous.

\item $K_{S}$ is USC.

\item $K_{S}=D_{S}$.

\item $K_{S}\left( x\right) $ is stable for every $x\in X$.
\end{enumerate}

Any one of these three conditions implies that $D_{S}$ is Hausdorff
continuous.
\end{corollary}

We now relate upper semicontinuity of $L_{\mathcal{F}}$ to eventual
stability. We need the following sequence of lemmas.

\begin{lemma}
\label{L1} Assume that $X$ is Hausdorff locally compact and $\mathcal{F}$ is
a filter basis on the connected subsets of $S$. For any neighborhood $U$ of $%
L_{\mathcal{F}}\left( x\right) $, there exists $A\in \mathcal{F}$ such that $%
Ax\subset U$.
\end{lemma}

\begin{proof}
There exists a neighborhood $W$ of $L_{\mathcal{F}}\left( x\right) $ such
that $W\subset U$ and $\mathrm{cls}\left( W\right) $ is compact. We claim
that $Ax\subset W$ for some $A\in \mathcal{F}$. Suppose by contradiction
that $Ax\nsubseteq W$ for all $A\in \mathcal{F}$. Since $W$ is a
neighborhood of $L_{\mathcal{F}}\left( x\right) $, we have $Ax\cap W\neq
\emptyset $ for all $A\in \mathcal{F}$. As $Ax$ is connected, it follows
that $Ax\cap \mathrm{fr}\left( \mathrm{cls}\left( W\right) \right) \neq
\emptyset $. For each $A\in \mathcal{F}$, take $t_{A}x\in Ax\cap \mathrm{fr}%
\left( \mathrm{cls}\left( W\right) \right) $. As $\mathrm{fr}\left( \mathrm{%
cls}\left( W\right) \right) $ is compact, we may assume that $%
t_{A}x\rightarrow y$ with $y\in \mathrm{fr}\left( \mathrm{cls}\left(
W\right) \right) $. Since $t_{A}\rightarrow _{\mathcal{F}}\infty $, this
means that $y\in L_{\mathcal{F}}\left( x\right) \cap \mathrm{fr}\left(
\mathrm{cls}\left( W\right) \right) $, which is impossible. Therefore $%
Ax\subset W$ for some $A\in \mathcal{F}$.
\end{proof}

\begin{proposition}
\label{T7}Assume that $X$ is Hausdorff locally compact and $\mathcal{F}$ is
a filter basis on the connected subsets of $S$ satisfying both hypotheses $%
\mathrm{H}_{1}$ and $\mathrm{H}_{3}$. Then $L_{\mathcal{F}}$ is USC on $L_{%
\mathcal{F}}\left( x\right) $ if and only if $L_{\mathcal{F}}\left( x\right)
$ is minimal and $\mathcal{F}$-eventually stable.
\end{proposition}

\begin{proof}
Suppose that $L_{\mathcal{F}}$ is USC on $L_{\mathcal{F}}\left( x\right) $.
For a given $y\in L_{\mathcal{F}}\left( x\right) $, we have $L_{\mathcal{F}%
}\left( y\right) \subset \mathrm{cls}\left( Sy\right) \subset L_{\mathcal{F}%
}\left( x\right) $, because $L_{\mathcal{F}}\left( x\right) $ is forward
invariant by hypothesis $\mathrm{H}_{1}$. To prove that $L_{\mathcal{F}%
}\left( x\right) $ is minimal, suppose by contradiction that $\mathrm{cls}%
\left( Sy\right) \varsubsetneq L_{\mathcal{F}}\left( x\right) $ and take $%
z\in L_{\mathcal{F}}\left( x\right) \setminus \mathrm{cls}\left( Sy\right) $%
. Then there is an open neighborhood $U$ of $\mathrm{cls}\left( Sy\right) $
such that $z\notin U$. Since $L_{\mathcal{F}}\left( x\right) $ is USC at $y$%
, there is an open neighborhood $V$ of $y$ such that $L_{\mathcal{F}}\left(
v\right) \subset U$ for every $v\in V$. As $y\in L_{\mathcal{F}}\left(
x\right) $, there is $s\in S$ such that $sx\in V$, and then $L_{\mathcal{F}%
}\left( sx\right) \subset U$. By hypothesis $\mathrm{H}_{3}$, it follows
that $z\in L_{\mathcal{F}}\left( x\right) \subset L_{\mathcal{F}}\left(
sx\right) \subset U$, which is a contradiction. Hence $L_{\mathcal{F}}\left(
x\right) =\mathrm{cls}\left( Sy\right) $, and therefore $L_{\mathcal{F}%
}\left( x\right) $ is minimal. We now show that $L_{\mathcal{F}}\left(
x\right) $ is $\mathcal{F}$-eventually stable. Let $U$ be an open
neighborhood of $L_{\mathcal{F}}\left( x\right) $. For a given $y\in L_{%
\mathcal{F}}\left( x\right) $, we have $L_{\mathcal{F}}\left( y\right)
\subset L_{\mathcal{F}}\left( x\right) \subset U$. Since $L_{\mathcal{F}}$
is USC at $y$, there is an open neighborhood $V_{y}$ of $y$ such that $L_{%
\mathcal{F}}\left( z\right) \subset U$ for all $z\in V_{y}$. Then $V=%
\underset{y\in L_{\mathcal{F}}\left( x\right) }{\bigcup }V_{y}$ is an open
neighborhood of $L_{\mathcal{F}}\left( x\right) $. If $z\in V$ then $L_{%
\mathcal{F}}\left( z\right) \subset U$. By Lemma \ref{L1} there is $A\in
\mathcal{F}$ such that $Az\subset U$. Hence $L_{\mathcal{F}}\left( x\right) $
is $\mathcal{F}$-eventually stable. As to the converse, take $y\in L_{%
\mathcal{F}}\left( x\right) $ and let $U$ be an open neighborhood of $L_{%
\mathcal{F}}\left( y\right) $. As $L_{\mathcal{F}}\left( x\right) $ is
minimal, we have $L_{\mathcal{F}}\left( x\right) =L_{\mathcal{F}}\left(
y\right) $. Take a compact neighborhood $V$ of $L_{\mathcal{F}}\left(
y\right) $ with $V\subset U$. As $L_{\mathcal{F}}\left( x\right) $ is $%
\mathcal{F}$-eventually stable, there is a neighborhood $W$ of $L_{\mathcal{F%
}}\left( x\right) $ such that for every $w\in W$ there is some $A\in
\mathcal{F}$ such that $Aw\subset V$, and hence $L_{\mathcal{F}}\left(
w\right) \subset \mathrm{cls}\left( Aw\right) \subset V\subset U$. Since $W$
is a neighborhood of $y$, $L_{\mathcal{F}}$ is USC at $y$.
\end{proof}

\begin{lemma}
\label{L2} Assume that $X$ is Hausdorff locally compact and $\mathcal{F}$ is
a filter basis on the connected subsets of $S$ satisfying both hypotheses $%
\mathrm{H}_{1}$ and $\mathrm{H}_{3}$. Then $L_{\mathcal{F}}$ is USC on $L_{%
\mathcal{F}}\left( x\right) $ if and only if $L_{\mathcal{F}}$ is USC on the
$\mathcal{F}$-domain of attraction $\mathfrak{A}\left( L_{\mathcal{F}}\left(
x\right) ,\mathcal{F}\right) $.
\end{lemma}

\begin{proof}
Suppose that $L_{\mathcal{F}}$ is USC on $L_{\mathcal{F}}\left( x\right) $.
For a given $y\in \mathfrak{A}\left( L_{\mathcal{F}}\left( x\right) ,%
\mathcal{F}\right) $ and a neighborhood $U$ of $L_{\mathcal{F}}\left(
y\right) $, take a compact neighborhood $V$ of $L_{\mathcal{F}}\left(
y\right) $ with $V\subset U$. By Proposition \ref{T7}, $L_{\mathcal{F}%
}\left( x\right) $ is minimal and $\mathcal{F}$-eventually stable. As $L_{%
\mathcal{F}}\left( y\right) \subset L_{\mathcal{F}}\left( x\right) $, it
follows that $L_{\mathcal{F}}\left( y\right) =L_{\mathcal{F}}\left( x\right)
$ and there is an open neighborhood $W$ of $L_{\mathcal{F}}\left( y\right) $
such that for every $w\in W$ there is $A\in \mathcal{F}$ with $Aw\subset V$.
Moreover, by Lemma \ref{L1}, there is $B\in \mathcal{F}$ such that $%
By\subset W$. Pick $s\in B$. As $sy\in W$, there is an open neighborhood $%
V_{y}$ of $y$ such that $sV_{y}\subset W$. For a given $z\in V_{y}$, there
is some $A\in \mathcal{F}$ such that $Asz\subset V$. Then we have $L_{%
\mathcal{F}}\left( z\right) \subset L_{\mathcal{F}}\left( sz\right) \subset
\mathrm{cls}\left( Asz\right) \subset V\subset U$. Hence $L_{\mathcal{F}}$
is USC at $y$, and therefore $L_{\mathcal{F}}$ is USC on $\mathfrak{A}\left(
L_{\mathcal{F}}\left( x\right) ,\mathcal{F}\right) $. The converse is
obvious since $L_{\mathcal{F}}\left( x\right) \subset \mathfrak{A}\left( L_{%
\mathcal{F}}\left( x\right) ,\mathcal{F}\right) $.
\end{proof}

\begin{lemma}
\label{L3} Assume that $\mathcal{F}$ satisfies the hypotheses $\mathrm{H}%
_{1} $ and $\mathrm{H}_{3}$. If $L_{\mathcal{F}}\left( y\right) $ is $%
\mathcal{F}$-eventually stable for every $y\in L_{\mathcal{F}}\left(
x\right) $ then $L_{\mathcal{F}}\left( x\right) $ is minimal.
\end{lemma}

\begin{proof}
Suppose by contradiction that $L_{\mathcal{F}}\left( x\right) $ is not
minimal. Since $L_{\mathcal{F}}\left( x\right) $ is compact and forward
invariant, there is a minimal set $M\varsubsetneq L_{\mathcal{F}}\left(
x\right) $. Take $y\in L_{\mathcal{F}}\left( x\right) \setminus M$. Then
there is an open neighborhood $U$ of $M$ such that $y\notin \mathrm{cls}%
\left( U\right) $. For $z\in M$, we have $L_{\mathcal{F}}\left( z\right) =M$%
, and since $L_{\mathcal{F}}\left( z\right) $ is $\mathcal{F}$-eventually
stable, there is a neighborhood $V$ of $M$ such that for every $v\in V$
there is some $A\in \mathcal{F}$ such that $Av\subset U$. As $M\subset L_{%
\mathcal{F}}\left( x\right) \cap V$, we can find $s\in S$ such that $sx\in V$%
, and then there is $A\in \mathcal{F}$ such that $Asx\subset U$. It follows
that $y\in L_{\mathcal{F}}\left( x\right) \subset L_{\mathcal{F}}\left(
sx\right) \subset \mathrm{cls}\left( Asx\right) \subset \mathrm{cls}\left(
U\right) $, which is impossible.
\end{proof}

We now have the following results on the upper semicontinuity of the limit
set function.

\begin{proposition}
\label{T8}Assume that $X$ is Hausdorff locally compact and $\mathcal{F}$ is
a filter basis on the connected subsets of $S$ satisfying both hypotheses $%
\mathrm{H}_{1}$ and $\mathrm{H}_{3}$. Then $L_{\mathcal{F}}$ is USC if and
only if $L_{\mathcal{F}}\left( x\right) $ is $\mathcal{F}$-eventually stable
for every $x\in X$.
\end{proposition}

\begin{proof}
Suppose that $L_{\mathcal{F}}\left( x\right) $ is $\mathcal{F}$-eventually
stable for all $x\in X$. By Lemma \ref{L3}, $L_{\mathcal{F}}\left( x\right) $
is minimal. By Proposition \ref{T7}, $L_{\mathcal{F}}$ is USC on $L_{%
\mathcal{F}}\left( x\right) $. By Lemma \ref{L2}, $L_{\mathcal{F}}$ is USC
on $\mathfrak{A}\left( L_{\mathcal{F}}\left( x\right) ,\mathcal{F}\right) $.
Since $x\in \mathfrak{A}\left( L_{\mathcal{F}}\left( x\right) ,\mathcal{F}%
\right) $, $L_{\mathcal{F}}$ is USC at $x$. Thus $L_{\mathcal{F}}$ is USC.
The converse follows by Proposition \ref{T7}.
\end{proof}

\begin{proposition}
Assume that $X$ is Hausdorff locally compact and $\mathcal{F}$ is a filter
basis on the connected subsets of $S$ satisfying both hypotheses $\mathrm{H}%
_{1}$ and $\mathrm{H}_{3}$. Assume that $\left( S,X\right) $ has global $%
\mathcal{F}$-attractor $\mathcal{A}$. Then $L_{\mathcal{F}}$ is USC if and
only if it is USC on $\mathcal{A}$.
\end{proposition}

\begin{proof}
Note that $X=\mathfrak{A}\left( \mathcal{A},\mathcal{F}\right) =\left\{ x\in
X:L_{\mathcal{F}}\left( x\right) \subset \mathcal{A}\right\} $. Then the
proof follows by Lemma \ref{L2}.
\end{proof}

In the next results, we discuss the Hausdorff continuity of the function $L_{%
\mathcal{F}}$. We firstly prove a technical lemma.

\begin{lemma}
\label{L4}Assume that $X$ is Hausdorff locally compact and $\mathcal{F}$ is
a filter basis on the connected subsets of $S$ satisfying both hypotheses $%
\mathrm{H}_{1}$ and $\mathrm{H}_{2}$. If $L_{\mathcal{F}}$ is LSC on $L_{%
\mathcal{F}}\left( x\right) $ then for each $y\in L_{\mathcal{F}}\left(
x\right) $ and a neighborhood $U$ of $L_{\mathcal{F}}\left( y\right) $ there
is a neighborhood $V$ of $L_{\mathcal{F}}\left( y\right) $ such that $Av\cap
U\neq \emptyset $ for all $v\in V$ and $A\in \mathcal{F}$. The converse
holds if $L_{\mathcal{F}}\left( y\right) $ is minimal for every $y\in L_{%
\mathcal{F}}\left( x\right) $.
\end{lemma}

\begin{proof}
Suppose that $L_{\mathcal{F}}$ is LSC on $L_{\mathcal{F}}\left( x\right) $
and take $y\in L_{\mathcal{F}}\left( x\right) $. For a given open
neighborhood $U$ of $L_{\mathcal{F}}\left( y\right) $ and $z\in L_{\mathcal{F%
}}\left( y\right) $, we have $L_{\mathcal{F}}\left( z\right) \subset L_{%
\mathcal{F}}\left( y\right) \subset U$. As $L_{\mathcal{F}}$ is LSC at $z$,
there is an open neighborhood $V_{z}$ of $z$ such that $U\cap L_{\mathcal{F}%
}\left( v\right) \neq \emptyset $ for all $v\in V_{z}$. Set $V=\underset{%
z\in L_{\mathcal{F}}\left( y\right) }{\bigcup }V_{z}$. Then $V$ is an open
neighborhood of $L_{\mathcal{F}}\left( y\right) $ and $U\cap L_{\mathcal{F}%
}\left( v\right) \neq \emptyset $ for all $v\in V$. This means that $U\cap
Av\neq \emptyset $ for all $v\in V$ for all $v\in V$ and $A\in \mathcal{F}$.
For the converse, assume that $L_{\mathcal{F}}\left( y\right) $ is minimal
for every $y\in L_{\mathcal{F}}\left( x\right) $. Let $y\in L_{\mathcal{F}%
}\left( x\right) $ and let $U$ be an open set with $U\cap L_{\mathcal{F}%
}\left( y\right) \neq \emptyset $. For $z\in L_{\mathcal{F}}\left( y\right) $%
, we have $L_{\mathcal{F}}\left( y\right) =\mathrm{cls}\left( Sz\right)
=K_{S}\left( z\right) $. Since $K_{S}$ is LSC at $z$, there is an open
neighborhood $V_{z}$ of $z$ such that $U\cap K_{S}\left( v\right) \neq
\emptyset $ for every $v\in V_{z}$. Take a neighborhood $W_{z}$ of $z$ such
that $\mathrm{cls}\left( W_{z}\right) $ is compact and $\mathrm{cls}\left(
W_{z}\right) \subset V_{z}$. We have the open covering $L_{\mathcal{F}%
}\left( y\right) \subset \underset{z\in L_{\mathcal{F}}\left( y\right) }{%
\bigcup }W_{z}$. By the compactness of $L_{\mathcal{F}}\left( y\right) $, we
can take a finite subcovering $L_{\mathcal{F}}\left( y\right) \subset
\bigcup_{i=1}^{n}W_{z_{i}}$. Set $W=\bigcup_{i=1}^{n}W_{z_{i}}$. Now there
is a neighborhood $V$ of $L_{\mathcal{F}}\left( y\right) $ such that $Av\cap
W\neq \emptyset $ for all $v\in V$ and $A\in \mathcal{F}$. By Lemma \ref{L1}%
, there is $A_{0}\in \mathcal{F}$ with $A_{0}y\subset V$. For $s\in A_{0}$,
there is an open neighborhood $N$ of $y$ such that $sN\subset V$. If $n\in N$
we have $Asn\cap W\neq \emptyset $ for all $A\in \mathcal{F}$. For each $%
A\in \mathcal{F}$, take $t_{A}\in A$ such that $t_{A}sn\in W$. Since $%
\mathrm{cls}\left( W\right) $ is compact, we may assume that $%
t_{A}sn\rightarrow w$ with $w\in \mathrm{cls}\left( W\right) $. As $%
t_{A}\rightarrow _{\mathcal{F}}\infty $, we have $w\in L_{\mathcal{F}}\left(
sn\right) $. Since $w\in \mathrm{cls}\left( W_{z_{i}}\right) \subset
V_{z_{i}}$ for some $z_{i}$, we have $U\cap K_{S}\left( w\right) \neq
\emptyset $, and then $\emptyset \neq U\cap \mathrm{cls}\left( Sw\right)
\subset U\cap L_{\mathcal{F}}\left( sn\right) $. By hypothesis $\mathrm{H}%
_{2}$, $L_{\mathcal{F}}\left( sn\right) \subset L_{\mathcal{F}}\left(
n\right) $ and hence $U\cap L_{\mathcal{F}}\left( n\right) \neq \emptyset $
whenever $n\in N$. Therefore $L_{\mathcal{F}}$ is LSC at $y$.
\end{proof}

Combining the hypotheses of Proposition \ref{T7} and Lemma \ref{L4}, we have
the following.

\begin{proposition}
\label{T9} Assume that $X$ is Hausdorff locally compact and $\mathcal{F}$ is
a filter basis on the connected subsets of $S$ satisfying hypotheses $%
\mathrm{H}_{1}$, $\mathrm{H}_{2}$, and $\mathrm{H}_{3}$. Then $L_{\mathcal{F}%
}$ is Hausdorff continuous on $L_{\mathcal{F}}\left( x\right) $ if and only
if $L_{\mathcal{F}}\left( x\right) $ is $\mathcal{F}$-eventually stable and
minimal.
\end{proposition}

\begin{proof}
Suppose that $L_{\mathcal{F}}\left( x\right) $ is $\mathcal{F}$-eventually
stable and minimal. By Proposition \ref{T7}, $L_{\mathcal{F}}$ is USC on $L_{%
\mathcal{F}}\left( x\right) $. It remains to show that $L_{\mathcal{F}}$ is
LSC on $L_{\mathcal{F}}\left( x\right) $. Indeed, for any $y\in L_{\mathcal{F%
}}\left( x\right) $, we have $L_{\mathcal{F}}\left( y\right) \subset L_{%
\mathcal{F}}\left( x\right) $, hence $L_{\mathcal{F}}$ is USC on $L_{%
\mathcal{F}}\left( y\right) $. By Proposition \ref{T7}, $L_{\mathcal{F}%
}\left( y\right) $ is minimal and $\mathcal{F}$-eventually stable. Now let $%
U $ be an open neighborhood of $L_{\mathcal{F}}\left( y\right) $. Then there
is an open neighborhood $V$ of $L_{\mathcal{F}}\left( y\right) $ such that
for each $v\in V$ there is $A_{0}\in \mathcal{F}$ such that $A_{0}v\subset U$%
. For any $A\in \mathcal{F}$, we have $A\cap A_{0}\neq \emptyset $, and then
$Av\cap U\neq \emptyset $. By Lemma \ref{L4}, $L_{\mathcal{F}}$ is LSC on $%
L_{\mathcal{F}}\left( x\right) $. The converse follows by Proposition \ref%
{T7}.
\end{proof}

We now present the main theorem on continuity of limit set function.

\begin{theorem}
\label{T10}Assume that $X$ is Hausdorff locally compact and $\mathcal{F}$ is
a filter basis on the connected subsets of $S$ satisfying hypotheses $%
\mathrm{H}_{1}$, $\mathrm{H}_{2}$, and $\mathrm{H}_{3}$. Then $L_{\mathcal{F}%
}$ is Hausdorff continuous if and only if $L_{\mathcal{F}}\left( x\right) $
is $\mathcal{F}$-eventually stable for every $x\in X$.
\end{theorem}

\begin{proof}
Suppose that $L_{\mathcal{F}}$ is $\mathcal{F}$-eventually stable for all $%
x\in X$. By Proposition \ref{T8}, $L_{\mathcal{F}}$ is USC. It remains to
show that $L_{\mathcal{F}}$ is LSC. Let $x\in X$. By Lemma \ref{L3}, $L_{%
\mathcal{F}}\left( x\right) $ is minimal. Let $U$ be an open set such that $%
U\cap L_{\mathcal{F}}\left( x\right) \neq \emptyset $. For every $y\in L_{%
\mathcal{F}}\left( x\right) $, we have $L_{\mathcal{F}}\left( y\right) =L_{%
\mathcal{F}}\left( x\right) $, hence $U\cap L_{\mathcal{F}}\left( y\right)
\neq \emptyset $. By Lemma \ref{L4}, $L_{\mathcal{F}}$ is LSC on $L_{%
\mathcal{F}}\left( x\right) $. Then there is a neighborhood $V_{y}$ of such
that $U\cap L_{\mathcal{F}}\left( v\right) \neq \emptyset $ for all $v\in
V_{y}$. Set $V=\underset{y\in L_{\mathcal{F}}\left( x\right) }{\bigcup }%
V_{y} $. Then $V$ is a neighborhood of $L_{\mathcal{F}}\left( x\right) $. By
Lemma \ref{L1}, $Ax\subset V$ for some $A\in \mathcal{F}$. Pick $s\in A$. As
$sx\in V$, there is a neighborhood $W$ of $x$ such that $sW\subset V$. If $%
w\in W$, we have $sw\in V$, and hence $U\cap L_{\mathcal{F}}\left( sw\right)
\neq \emptyset $. As $L_{\mathcal{F}}\left( sw\right) \subset L_{\mathcal{F}%
}\left( w\right) $, it follows that $U\cap L_{\mathcal{F}}\left( w\right)
\neq \emptyset $. Thus $L_{\mathcal{F}}$ is LSC at $x$. The converse follows
by Proposition \ref{T9}.
\end{proof}

Proposition \ref{T9} and Theorem \ref{T10} together imply the following.

\begin{theorem}
\label{T11} Assume that $X$ is Hausdorff locally compact and $\mathcal{F}$
is a filter basis on the connected subsets of $S$ satisfying hypotheses $%
\mathrm{H}_{1}$, $\mathrm{H}_{2}$, and $\mathrm{H}_{3}$. Assume that $\left(
S,X\right) $ has global $\mathcal{F}$-attractor $\mathcal{A}$. Then $L_{%
\mathcal{F}}$ is Hausdorff continuous if and only if it is Hausdorff
continuous on $\mathcal{A}$.
\end{theorem}

A subset $Y\subset X$ is usually called $\mathcal{F}$-topologically
transitive if $Y=\omega \left( y,\mathcal{F}\right) $ for some $y\in Y$. As
an immediate consequence from Theorem \ref{T11} we have the following.

\begin{corollary}
Assume that $X$ is Hausdorff locally compact and $\mathcal{F}$ is a filter
basis on the connected subsets of $S$ satisfying hypotheses $\mathrm{H}_{1}$%
, $\mathrm{H}_{2}$, and $\mathrm{H}_{3}$. Assume that $\left( S,X\right) $
has global $\mathcal{F}$-asymptotically stable set $\mathcal{A}$, that is, $%
\mathcal{A}$ is the global $\mathcal{F}$-attractor and is stable. If $%
\mathcal{A}$ is $\mathcal{F}$-topologically transitive then $L_{\mathcal{F}}$
is Hausdorff continuous.
\end{corollary}

It should be observed that the existence of the global $\mathcal{F}$%
-asymptotically stable set is not sufficient for Hausdorff continuity. See
Example \ref{Ex2} in the next section.

\section{\label{Examples}Examples}

In this last section we provide illustrating examples for the setting of
this paper.

\subsection{Homogeneous spaces}

Let $G$ be a topological group and $\mathcal{V}$ a basis of symmetric
neighborhoods at the identity $e$ of $G$. For each $V\in \mathcal{V}$,
define the open covering of $G$
\begin{equation*}
\mathcal{U}_{V}=\left\{ Vg:g\in G\right\} .
\end{equation*}%
Let $\mathcal{O}$ be the family of all open covering $\mathcal{U}_{V}$, $%
V\in \mathcal{V}$. This family is admissible (see \cite[Section 4]{Souza})
and we have
\begin{equation*}
\mathrm{St}\left[ g,\mathcal{U}_{V}\right] =V^{2}g,\text{\qquad for every }%
g\in G.
\end{equation*}%
In general, a topological group is not metrizable.

Let $S\subset G$ be a closed subgroup of $G$ and $\left. G\right/ S$ the set
of all left coset $gS$. Then $\left. G\right/ S$ is a subset of the
hyperspace $\mathcal{H}\left( G\right) $. Now consider the standard left
action of $S$ on $G$. We have $K_{S}\left( g\right) =gS$ for every $g\in G$.
Then the function $K_{S}:G\rightarrow \mathcal{H}\left( G\right) $ has
values in $\left. G\right/ S$ and corresponds to the standard projection $%
\pi :G\rightarrow \left. G\right/ S$.

Provide $\left. G\right/ S$ with the quotient topology. Then $\pi
:G\rightarrow \left. G\right/ S$ is an open continuous map. The following
theorem shows that the quotient topology and the uniform topology induced
from $\mathcal{H}\left( G\right) $ are equivalent.

\begin{theorem}
\label{T12} The quotient topology on $\left. G\right/ S$ coincides with the
uniform topology induced from $\mathcal{H}\left( G\right) $ and the quotient
map $\pi :G\rightarrow \left. G\right/ S$ is uniformly continuous.
\end{theorem}

\begin{proof}
By Theorem \ref{T3}, it is enough to compare the open sets of $\left.
G\right/ S$ and the sets $\mathrm{B}_{H}\left( gS,\mathcal{U}_{V}\right)
\cap \left. G\right/ S$. We claim that $\pi \left( \mathrm{St}\left[ g,%
\mathcal{U}_{V}\right] \right) =\mathrm{B}_{H}\left( gS,\mathcal{U}%
_{V}\right) \cap \left. G\right/ S$ for all $g\in G$ and $\mathcal{U}_{V}\in
\mathcal{O}$. Firstly, by considering $gS$ as a subset of $G$, we have $%
\mathrm{St}\left[ gS,\mathcal{U}_{V}\right] =V^{2}gS$. Note also that $%
g_{1}S\in \pi \left( \mathrm{St}\left[ g,\mathcal{U}_{V}\right] \right) =\pi
\left( V^{2}g\right) $ if and only if $g_{1}S\subset V^{2}gS$. By the
symmetry of $V$, $g_{1}S\subset V^{2}gS$ if and only if $gS\subset
V^{2}g_{1}S$. Hence $g_{1}S\in \pi \left( \mathrm{St}\left[ g,\mathcal{U}_{V}%
\right] \right) $ if and only if $g_{1}S\in \mathrm{B}_{H}\left( gS,\mathcal{%
U}_{V}\right) $, and therefore $\pi \left( \mathrm{St}\left[ g,\mathcal{U}%
_{V}\right] \right) =\mathrm{B}_{H}\left( gS,\mathcal{U}_{V}\right) \cap
\left. G\right/ S$, proving the claim. This means that $\pi $ is uniformly
continuous with respect to $\mathcal{O}$ and the covering uniformity induced
from $\mathcal{H}\left( G\right) $. Now, since $\pi $ is an open map, $%
\mathrm{B}_{H}\left( gS,\mathcal{U}_{V}\right) \cap \left. G\right/ S$ is an
open set of the quotient topology on $\left. G\right/ S$, for all $g\in G$
and $\mathcal{U}_{V}\in \mathcal{O}$. On the other hand, let $U\subset
\left. G\right/ S$ be an open set. Then $\pi ^{-1}\left( U\right) $ is an
open set in $G$. Take $g\in \pi ^{-1}\left( U\right) $. There is $\mathcal{U}%
_{V}$ such that $\mathrm{St}\left[ g,\mathcal{U}_{V}\right] \subset \pi
^{-1}\left( U\right) $, and then $\mathrm{B}_{H}\left( gS,\mathcal{U}%
_{V}\right) \cap \left. G\right/ S=\pi \left( \mathrm{St}\left[ g,\mathcal{U}%
_{V}\right] \right) \subset U$. This finalize the proof.
\end{proof}

Thus the function $K_{S}:G\rightarrow \mathcal{H}\left( G\right) $ is
Hausdorff continuous. This means that $\rho _{H}\left( g_{\lambda
}S,gS\right) \rightarrow \mathcal{O}$ whenever $g_{\lambda }\rightarrow g$.
If $S$ is compact, each left coset $gS$ is stable, by Theorem \ref{T1}.

\subsection{Transformation semigroups on function spaces}

Let $E$ be a normed vector space endowed with the admissible family $%
\mathcal{O}_{\mathrm{d}}$ as stated in Example \ref{Ex1}. For finite
sequences $\alpha =\left\{ x_{1},...,x_{k}\right\} $ in $E$ and $\epsilon
=\left\{ \mathcal{U}_{\varepsilon _{1}},...,\mathcal{U}_{\varepsilon
_{k}}\right\} $ in $\mathcal{O}_{\mathrm{d}}$, let $\mathcal{U}_{\alpha
}^{\epsilon }$ be the cover of $E^{E}$ given by the sets of the form $%
\tprod_{x\in E}U_{x}$ where $U_{x_{i}}=\mathrm{B}\left( a_{i},\varepsilon
_{i}\right) \in \mathcal{U}_{\varepsilon _{i}}$, for $i=1,...,k$, and $%
U_{x}=E$ otherwise. The family $\mathcal{O}_{\mathrm{p}}=\left\{ \mathcal{U}%
_{\alpha }^{\epsilon }\right\} $ is a base for the uniformity of pointwise
convergence on $E^{E}$ (see e.g. \cite[Corollary 37.13]{Will}).

Let $E^{E}$ endowed with the pointwise convergence topology. Then the
inclusion map $i:E\hookrightarrow E^{E}$, where $i\left( x\right) $ is the
constant function $i\left( x\right) \equiv x$, is a continuous map. It is
well-known that $E^{E}$ is not metrizable with the pointwise convergence
topology.

\begin{example}
\label{Ex4}Let $\mathbb{N}$ be the semigroup of positive integers with
multiplication and $\mu :\mathbb{N}\times E^{E}\rightarrow E^{E}$ the action
given by $\mu \left( n,f\right) =f^{n}$. Consider the filter basis $\mathcal{%
F}=\left\{ A_{n}:n\in \mathbb{N}\right\} $, where $A_{n}=\left\{ k\in
\mathbb{N}:k\geq n\right\} $. In this case, $n_{\lambda }\rightarrow _{%
\mathcal{F}}\infty $ in $\mathbb{N}$ means $n_{\lambda }\rightarrow +\infty $%
. Let $F\subset E$ be a compact set and $X\subset E^{E}$ be the subspace of
all contraction maps of $E$ with fixed point in $F$ and same Lipschitz
constant $L<1$. Then $X$ is forward invariant and $i\left( F\right) \subset
X $ is a compact, closed, and forward invariant set in $X$. Consider the
restriction action $\mu :\mathbb{N}\times X\rightarrow X$. We have $J_{%
\mathcal{F}}\left( X\right) =i\left( F\right) $, and therefore $i\left(
F\right) $ is the global $\mathcal{F}$-asymptotically stable set for $\left(
\mathbb{N},X\right) $ (the details will appear latter in \cite{RichardB}).
Note that $L_{\mathcal{F}}\left( f\right) =\left\{ i\left( x_{f}\right)
\right\} $, where $x_{f}$ denotes the fixed point of $f$. In particular, $L_{%
\mathcal{F}}\left( i\left( x\right) \right) =\left\{ i\left( x\right)
\right\} $, for every $x\in F$, and therefore $L_{\mathcal{F}}$ is Hausdorff
continuous on $i\left( F\right) $. Unfortunately, we can not apply Theorem %
\ref{T11} here to conclude that $L_{\mathcal{F}}$ is Hausdorff continuous on
$X$, since $A_{n}$ is desconnected for all $n\in \mathbb{N}$. Then we
provide a direct proof for it. Firstly, note that upper semicontinuity and
lower semicontinuity are equivalent in this case, since each limit set
consists of a single point. Then let $f\in X$ and $\mathcal{U}_{\alpha
}^{\epsilon }\in \mathcal{O}_{\mathrm{p}}$, with $\alpha =\left\{
x_{1},...,x_{k}\right\} $ and $\epsilon =\left\{ \mathcal{U}_{\varepsilon
_{1}},...,\mathcal{U}_{\varepsilon _{k}}\right\} $. Define $\delta =\dfrac{%
\varepsilon \left( 1-L\right) }{2}$, where $\varepsilon =\min \left\{
\varepsilon _{1},...,\varepsilon _{k}\right\} $, and take $g\in \mathrm{St}%
\left[ f,\mathcal{U}_{\left\{ x_{f}\right\} }^{\left\{ \delta \right\} }%
\right] $. Then $g,f\in \tprod_{x\in E}U_{x}$ where $U_{x_{f}}=\mathrm{B}%
\left( a,\delta \right) \in \mathcal{U}_{\delta }$ and $U_{x}=E$ otherwise.
Hence we have
\begin{eqnarray*}
\left\Vert x_{g}-x_{f}\right\Vert &\leq &\left\Vert g\left( x_{g}\right)
-g\left( x_{f}\right) \right\Vert +\left\Vert g\left( x_{f}\right) -f\left(
x_{f}\right) \right\Vert \\
&<&L\left\Vert x_{g}-x_{f}\right\Vert +2\delta
\end{eqnarray*}%
and thus $\left\Vert x_{g}-x_{f}\right\Vert <\dfrac{2\delta }{1-L}%
=\varepsilon $. This means that
\begin{equation*}
\left\Vert i\left( x_{g}\right) \left( x_{i}\right) -x_{f}\right\Vert
=\left\Vert x_{g}-x_{f}\right\Vert <\varepsilon \leq \varepsilon _{i}
\end{equation*}%
for every $x_{i}\in \alpha $. Hence $i\left( x_{g}\right) \left(
x_{i}\right) \in \mathrm{B}\left( x_{f},\varepsilon _{i}\right) $, for any $%
x_{i}\in \alpha $, and therefore $i\left( x_{g}\right) \in \mathrm{St}\left[
i\left( x_{f}\right) ,\mathcal{U}_{\alpha }^{\epsilon }\right] $ for every $%
g\in \mathrm{St}\left[ f,\mathcal{U}_{\left\{ x_{f}\right\} }^{\left\{
\delta \right\} }\right] $. This means that $L_{\mathcal{F}}$ is Hausdorff
continuous at $f$.
\end{example}

\subsection{Multi-time dynamical systems}

An $n$-time dynamical system is an action of a convex cone $S\subset \mathbb{%
R}^{n}$ on a topological space $X$. Take a nonzero vector $u\in S$ and
consider the family of translates
\begin{equation*}
\mathcal{F}=\left\{ S+tu:t\geq 0\right\} .
\end{equation*}%
Then $\mathcal{F}$ is a filter basis on the connected subsets of $S$
satisfying hypotheses $\mathrm{H}_{1}$, $\mathrm{H}_{2}$, and $\mathrm{H}%
_{3} $, and the limit behavior with respect to $\mathcal{F}$ means the limit
behavior on the direction of $u$.

\begin{example}
Consider the $2$-time dynamical system $\left( S,\mathbb{R}^{2}\right) $
where $S=\left\{ \left( s,t\right) \in \mathbb{R}^{2}:s,t\geq 0\right\} $
and $\left( s,t\right) \left( x,y\right) =\left( b^{s}x,b^{t}y\right) $,
with $b$ a fixed real number in the interval $\left( 0,1\right) $. Consider
the filter basis $\mathcal{F}=\left\{ S+te_{1}:t\geq 0\right\} $ where the
vector $e_{1}$ determines the direction of the axis $0x$. In this case, $%
\left( s_{\lambda },t_{\lambda }\right) \rightarrow _{\mathcal{F}}\infty $
means $s_{\lambda }\rightarrow +\infty $. For any $\left( x,y\right) \in
\mathbb{R}^{2}$, we have $L_{\mathcal{F}}\left( x,y\right) =\left\{ \left(
0,z\right) :0\leq z\leq by\right\} $. Hence the axis $0y$ is the global $%
\mathcal{F}$-attractor of $\left( S,\mathbb{R}^{2}\right) $. It is easily
seen that $L_{\mathcal{F}}\left( x,y\right) $ is stable, and then it is $%
\mathcal{F}$-eventually stable. By Theorem \ref{T10}, $L_{\mathcal{F}}$ is
Hausdorff continous.
\end{example}

\subsection{Control systems}

Consider a control system $\dot{x}=X\left( x,u\left( t\right) \right) $ on a
connected $C^{\infty }$-Riemannian manifold $M$, with control range $%
U\subset \mathbb{R}^{n}$ and piecewise control functions $\mathcal{U}%
_{pc}=\left\{ u:\mathbb{R}\rightarrow U:u\text{ piecewise constant}\right\} $%
. Assume that, for each $u\in \mathcal{U}_{pc}$ and $x\in M$, the preceding
equation has a unique solution $\varphi \left( t,x,u\right) $, $t\in \mathbb{%
R}$, with $\varphi \left( 0,x,u\right) =x$, and the vector fields $X\left(
\cdot ,u\right) $, $u\in U$, are complete. Set $F=\left\{ X\left( \cdot
,u\right) :~u\in U\right\} $. The system semigroup $S$ is defined as
\begin{equation*}
S=\left\{ \mathrm{e}^{t_{n}Y_{n}}\mathrm{e}^{t_{n-1}Y_{n-1}}...\mathrm{e}%
^{t_{0}Y_{0}}:Y_{j}\in F,t_{j}\geq 0,n\in \mathbb{N}\right\} .
\end{equation*}%
The family of vector fields $F$ and the system semigroup determine the
trajectories of the control system in the sense that $Sx=\left\{ \varphi
\left( t,x,u\right) :t\geq 0,u\in \mathcal{U}_{pc}\right\} $ for every $x\in
M$.

For $t>0$ we define the set
\begin{equation*}
S_{\geq t}=\left\{ \mathrm{e}^{t_{n}Y_{n}}\mathrm{e}^{t_{n-1}Y_{n-1}}...%
\mathrm{e}^{t_{0}Y_{0}}:Y_{j}\in F,t_{j}\geq 0,\overset{n}{\underset{j=0}{%
\sum }}t_{j}\geq t,n\in \mathbb{N}\right\}
\end{equation*}%
and take the family $\mathcal{F}=\left\{ S_{\geq t}:t>0\right\} $. This
family is a filter basis on the connected subsets of $S$ and satisfies both
hypotheses $\mathrm{H}_{1}$ and $\mathrm{H}_{2}$ but need not satisfy
hypothesis $\mathrm{H}_{3}$. Note that $\varphi \left( t_{\lambda
},x_{\lambda },u_{\lambda }\right) \rightarrow _{\mathcal{F}}\infty $ means $%
t_{\lambda }\rightarrow +\infty $.

We consider a Riemannian distance $\mathrm{d}$ on $M$ and endow $M$ with the
admissible family $\mathcal{O}_{\mathrm{d}}$ as stated in Example \ref{Ex1}.

\begin{example}
Consider the control system
\begin{equation*}
\dot{x}=\left(
\begin{array}{ll}
-u\left( t\right) & 1 \\
-1 & 0%
\end{array}%
\right) x\left( t\right) ,\quad U=\left[ 0,1\right] ,
\end{equation*}%
on $\mathbb{R}^{2}$. For $u\equiv 0$ the system moves on circles centered at
$0$; for $u>0$ the system moves on spirals centered at $0$. For any $u\in
\mathbb{R}^{2}$ we have
\begin{equation*}
K_{S_{\geq t}}\left( u\right) =K_{S}\left( u\right) =\left\{ v\in \mathbb{R}%
^{2}:\left\Vert v\right\Vert \leq \left\Vert u\right\Vert \right\}
\end{equation*}%
for all $t>0$, and then $L_{\mathcal{F}}\left( u\right) =K_{S}\left(
u\right) =\left\{ v\in \mathbb{R}^{2}:\left\Vert v\right\Vert \leq
\left\Vert u\right\Vert \right\} $. Let $\varepsilon >0$ and take $w\in
\mathrm{B}\left( u,\varepsilon \right) $. Then we have $L_{\mathcal{F}%
}\left( w\right) =\left\{ v\in \mathbb{R}^{2}:\left\Vert v\right\Vert \leq
\left\Vert w\right\Vert \right\} $. If $v\in L_{\mathcal{F}}\left( w\right) $%
, it follows that $\left\Vert v\right\Vert \leq \left\Vert w-u\right\Vert
+\left\Vert u\right\Vert <\varepsilon +\left\Vert u\right\Vert $, and hence $%
v\in \mathrm{B}\left( L_{\mathcal{F}}\left( u\right) ,\varepsilon \right) $.
This means that $L_{\mathcal{F}}\left( w\right) \subset \mathrm{B}\left( L_{%
\mathcal{F}}\left( u\right) ,\varepsilon \right) $ and therefore $L_{%
\mathcal{F}}$ is USC. Since $L_{\mathcal{F}}=K_{S}$ is LSC, $L_{\mathcal{F}}$
is Hausdorff continuous.
\end{example}

\begin{example}
\label{Ex2}Consider the control system
\begin{equation*}
\dot{x}\left( t\right) =X_{0}\left( x\left( t\right) \right) +u\left(
t\right) X_{1}\left( x\left( t\right) \right) ,\qquad U=\left[ 1,2\right] ,
\end{equation*}%
on $\mathbb{R}^{2}$, where $X_{0},X_{1}$ are vector fields given by
\begin{equation*}
X_{0}\left( x,y\right) =\left( y,-x\right) \text{\quad and\quad }X_{1}\left(
x,y\right) =\left( x-xy^{2}-x^{3},y-yx^{2}-y^{3}\right) .
\end{equation*}%

The unit disk $\mathbb{D}^{1}=\left\{ x\in \mathbb{R}^{2}:\left\Vert
x\right\Vert \leq 1\right\} $ is the global $\mathcal{F}$-asymptotically
stable set of the system (see \cite[Example 7]{Ra} for details).
Nevertheless, $L_{\mathcal{F}}\left( x\right) $ is the unit sphere $\mathbb{S%
}^{1}$, if $\left\Vert x\right\Vert \geq 1$, and $L_{\mathcal{F}}\left(
x\right) $ is the origin $0$, if $\left\Vert x\right\Vert <1$. Hence $L_{%
\mathcal{F}}$ is not USC.
\end{example}

\end{document}